\documentclass[review]{siamart1116}

\usepackage{lipsum}
\usepackage{amsfonts}
\usepackage{graphicx}
\usepackage{epstopdf}
\usepackage{algorithmic}
\ifpdf
  \DeclareGraphicsExtensions{.pdf,.pdf,.png,.jpg}
\else
  \DeclareGraphicsExtensions{.pdf}
\fi

\numberwithin{theorem}{section}

\newcommand{\TheTitle}{Numerical Integration in Multiple Dimensions with Designed Quadrature}
\newcommand{\TheAuthors}{V. Keshavarzzadeh, R. M. Kirby, and A. Narayan}

\headers{Designed quadrature}{\TheAuthors}

\title{{\TheTitle}\thanks{Accepted for publication in SIAM Journal on Scientific Computing - Methods and Algorithms for Scientific Computing section. \funding{This research was sponsored by ARL under Cooperative Agreement Number W911NF-12-2-0023. The views and conclusions contained in this document are those of the authors and should not be interpreted as representing the official policies, either expressed or implied, of ARL or the U.S. Government. The U.S. Government is authorized to reproduce and distribute reprints for Government purposes notwithstanding any copyright notation herein. The first and third authors are partially supported by AFOSR FA9550-15-1-0467. The third author is partially supported by DARPA EQUiPS N660011524053.}}}

\author{
  Vahid Keshavarzzadeh\thanks{Scientific Computing and Imaging Institute, University of Utah, Salt Lake City, UT (\email{vkeshava@sci.utah.edu}, \email{kirby@sci.utah.edu}, \email{akil@sci.utah.edu}).}
  \and
  Robert M. Kirby\footnotemark[2]\hskip 3pt$^,$\thanks{School of Computing, University of Utah, Salt Lake City, UT}
  \and
  Akil Narayan\footnotemark[2]\hskip 3pt$^,$\thanks{Department of Mathematics, University of Utah, Salt Lake City, UT}
}

\usepackage{amsopn}

\newsiamremark{remark}{Remark}

\ifpdf
\hypersetup{
  pdftitle={\TheTitle},
  pdfauthor={\TheAuthors}
}
\fi

\usepackage{bbm}

\usepackage{mathrsfs,mathtools}

\usepackage{bm}

\usepackage{enumitem}
\usepackage{mathdefs}

\newcommand{\annote}[1]{#1}
\newcommand{\revv}[1]{#1}
\newcommand{\bs}[1]{\boldsymbol{#1}}

\nolinenumbers

\begin{document}

\maketitle

\begin{abstract}
  We present a systematic computational framework for generating positive quadrature rules in multiple dimensions on general geometries. A direct moment-matching formulation that enforces exact integration on polynomial subspaces yields nonlinear conditions and geometric constraints on nodes and weights. We use penalty methods to address the geometric constraints, and subsequently solve a quadratic minimization problem via the Gauss-Newton method. Our analysis provides guidance on requisite sizes of quadrature rules for a given polynomial subspace, and furnishes useful user-end stability bounds on error in the quadrature rule in the case when the polynomial moment conditions are violated by a small amount due to, e.g., finite precision limitations or stagnation of the optimization procedure. We present several numerical examples investigating optimal low-degree quadrature rules, Lebesgue constants, and 100-dimensional quadrature. Our capstone examples compare our quadrature approach to popular alternatives, such as sparse grids and quasi-Monte Carlo methods, for problems in linear elasticity and topology optimization.
\end{abstract}



\section{Introduction}

Numerical quadrature, the process of computing approximations to integrals, is widely used in many fields of science and engineering. A convenient and popular choice is a quadrature rule that uses point evaluations of a function $f$:
\begin{align*}
  \int_\Gamma f(\bs{x}) \omega(\bs{x}) \dx{x} \approx \sum_{j=1}^n f(\bs{x}_j) w_j,
\end{align*}
where $\Gamma$ is some set in $d$-dimensional Euclidean space $\R^d$, $\omega$ is a positive weight function, and $\bs{x}_j$ and $w_j$ are the nodes and weights, respectively, of the quadrature rule that must be determined. The main desirable properties of quadrature rules are accuracy for a broad class of functions, a small number $n$ of nodes/weights, and positivity of the weights. \annote{(Positive weights are desired so that the absolute condition number of the quadrature rule is controlled.)}

In one dimension, Gaussian quadrature rules \cite{Ma96,szego_orthogonal_1975} satisfy many of these desirable properties, but computing an efficient quadrature rule (\annote{or ``cubature'' rule}) for higher dimensions is a considerably more challenging problem. When $\Gamma$ and $\omega$ are of tensor-product form, one straightforward construction results from tensorization of univariate quadrature rules. However, the computational complexity required to evaluate $f$ at the nodes of a tensorized quadrature rule quickly succumbs to the curse of dimensionality.

Substantial progress has been made in constructing attractive multivariate quadrature rules. Sparse grids rely on a sophisticated manipulation of univariate quadrature rules \cite{Bungartz04,gerstner_numerical_1998}. Quasi-Monte Carlo methods generate sequences that have low-discrepancy properties \cite{Niederreiter92,Owen03,sloan_when_1998}. Mathematical characterizations of quadrature rules with specified exactness on polynomial spaces yield efficient nodes and weights \cite{Bos06,caliari_bivariate_2005,stroud_remarks_1957,xu_characterization_1994}.

The main contribution of this paper is a systematic computational approach for designing multivariate quadrature rules with exactness on general finite-dimensional polynomial spaces.  \annote{Using polynomial exactness as a desideratum for constructing quadrature rules is not the only approach one could use (e.g., \revv{quasi-Monte Carlo} methods do not adopt this approach). \revv{However, when the integrand $f$ can be accurately approximated by a polynomial expansion with a small number of significant terms, then approximating the integral with a quadrature rule that is designed to integrate the significant terms can be very efficient \cite{canuto_spectral_2011,canuto_spectral_2014}}. In particular, finite-dimensional polynomial spaces can well-approximate solutions to some parametric operator equations \cite{cohen_convergence_2010}, and empirical tests with many engineering problems show that polynomial approximations are very efficient \cite{agarwal_domain_2009,babuska_stochastic_2010,bungartz_multivariate_2003}.}

Our computational approach revolves around optimization; many algorithms for computing nodal sets via optimization have already been proposed \cite{Ma96,mousavi_generalized_2010,Ryu15,taylor_cardinal_2007,taylor_algorithm_2000,van_barel_approximating_2014,Xiao10}.  Our method, which we call \textit{designed quadrature}, has the following advantages:
\begin{itemize}
  \item we can successfully compute nodal sets in up to $100$ dimensions;
  \item positivity of the weights is ensured;
  \item quadrature rules over non-standard geometries can be computed; and
  \item a prescribed polynomial accuracy can be sought over general polynomial spaces, not restricted to, e.g., total degree spaces.
\end{itemize}
Our approach is simple: we formulate moment-matching conditions and geometric constraints that prescribe nonlinear conditions on the nodes and weights. This direct formulation allows significant flexibility with respect to geometry, weight function $\omega$, and polynomial accuracy. \annote{Indeed, our procedures can compute quadrature rules with hyperbolic cross polynomial spaces, see Section \ref{sec:results-highd}, and can constrain nodal locations to awkward geometries, see Section \ref{sec:results-U}.} Our computational approach is to use constrained optimization algorithms to compute a quadrature rule from the moment-matching conditions. Our mathematical analysis provides a stability bound on error of the quadrature rule if the moment-matching conditions are violated (e.g., due to numerical finite precision). We apply our designed quadrature rules to several realistic problems in computational science, including problems in linear elasticity and topology optimization. Comparisons against competing methods, such as sparse grids and low-discrepancy sequences, illustrate that designed quadrature often attains superior accuracy with many fewer nodes.

Our procedure is not without shortcomings: Being a direct moment-matching problem, our framework relies on large-scale optimization in high dimensions. For a specified polynomial subspace on which we require integration accuracy, we cannot \textit{a priori} determine the number of nodes that our procedure will produce (although we review some theory that provides upper and lower bounds for $n$). \annote{We likewise cannot ensure that our algorithm produces an optimal quadrature rule size, but our numerical results suggest favorable comparison with alterative techniques, see Section \ref{sec:results-comparison}.} Some of the optimization tools we use have tunable parameters; we have made automated choices for these parameters but leave to future work to prove that the algorithm performs well for arbitrary dimensions, weight functions, or polynomial spaces.

This paper is organized as follows. In Section \ref{sec:math} we discuss the mathematical setting and formulate the optimization problem. This section also presents theory for the requisite number of nodes and stability of quadrature rules for approximate moment-matching. Section \ref{S3} details the computational framework for generating designed quadrature rules. Numerical results are shown in Section \ref{sec:results}.

\section{Multivariate Quadrature}\label{sec:math}

\subsection{Notation}

Let $\omega$ be a given \revv{non-negative} weight function (e.g., a probability density function) whose support is $\Gamma \subset \R^d$, where $d \geq 1$ and $\Gamma$ need not be compact. A point $\bs{x} \in \R^d$ has components $\bs{x} = \left( x^{(1)}, x^{(2)}, \ldots, x^{(d)} \right)$. The space $L^2_\omega(\Gamma)$ is the set of functions $f$ defined by
\begin{align*}
  L^2_\omega(\Gamma) &= \left\{ f: \Gamma \rightarrow \R \; \big| \; \|f \| < \infty \right\}, &
  \left\| f \right\|^2 &= \left( f, f \right), &
  \left( f, g \right) &= \int_\Gamma f(\bs{x}) g(\bs{x}) \omega(\bs{x}) \dx{\bs{x}}.
\end{align*}
We use standard multi-index notation: $\bs{\alpha} \in \N_0^d$ denotes a multi-index, and $\Lambda$ a collection of multi-indices. We have
\begin{align*}
  \bs{\alpha} &= (\alpha_1, \ldots, \alpha_d), & \bs{x}^{\bs{\alpha}} &= \prod_{j=1}^d \left( x^{(j)} \right)^{\alpha_j}, & |\bs{\alpha}| = \sum_{j=1}^d \alpha_j.
\end{align*}
We impose a partial ordering on multi-indices via component-wise comparisons: with $\bs{\alpha}$, $\bs{\beta} \in \N_0^d$, then $\bs{\alpha} \leq \bs{\beta}$ if and only if all component-wise inequalities are true. A multi-index set $\Lambda$ is called \textit{downward closed} if
\begin{align*}
  \bs{\alpha} \in \Lambda \hskip 10pt \Longrightarrow \hskip 10pt \bs{\beta} \in \Lambda \hskip 15pt \forall \; \bs{\beta} \leq \bs{\alpha}.
\end{align*}
We assume throughout this paper that the weight function has finite polynomial moments of all orders:
\begin{align*}
  \int_\Gamma \left(\bs{x}^{\bs \alpha}\right)^2 \omega(\bs{x}) &< \infty, & \bs \alpha &\in \N_0^d.
\end{align*}
This assumption ensures existence of polynomial moments. Our ultimate goal is to construct a set of $n$ points $\left\{ \bs{x}_q \right\}_{q=1}^n \subset \Gamma$ and positive weights $w_q > 0$ such that
\begin{subequations}
\begin{align}\label{eq:quadrature-approx}
  I(f) = \int_\Gamma f(\bs{x}) \omega(\bs{x}) \dx{x} \approx \sum_{q=1}^n w_q f(\bs{x}_q),
\end{align}
for functions $f$ within a ``large" class of functions. We attempt to achieve this by enforcing equality above for $f$ in a subspace $\Pi$ of polynomials:
\begin{align}\label{eq:quadrature-equality}
  \int_\Gamma f(\bs{x}) \omega(\bs{x}) \dx{x} &= \sum_{q=1}^n w_q f(\bs{x}_q), & f &\in \Pi.
\end{align}
\end{subequations}
\annote{\revv{The quadrature strategy is accurate if $f$ can be well-approximated by a polynomial from $\Pi$.} There are numerous technical conditions on $\Pi$ and $f$ that yield quantitative statements about polynomial approximation accuracy, e.g., \cite{bernardi_spectral_1997}. In this article, we assume that $\Pi$ is given and fixed through some \textit{a priori} study ensuring that there exists a polynomial in $\Pi$ that accurately approximates $f$ to within some user-specified tolerance.} Typically we will define $\Pi$ through some finite multi-index set $\Lambda$:
\begin{align*}
  \Pi = \mathrm{span} \left\{ \bs{x}^ {\bs \alpha} \;\; \big| \;\; \bs \alpha \in \Lambda \right\}.
\end{align*}

In many applications, the function $f$ typically exhibits smoothness (e.g., integrable high-order derivatives), which in turn implies that polynomial approximations converge at a high order with respect to the degree of approximation. Under the assumption that $f$ is smooth, we therefore expect that the integral of a polynomial that approximates $f$ to be a good approximation if the approximating polynomial space $\Pi$ contains high-degree polynomials. Our main goal in this paper is then familiar when viewed through the lens of classical analysis: make $\Pi$ as large as possible while keeping $n$ as small as possible.

Two particularly popular choices for polynomial spaces $\Pi$ can be defined by the index sets
\begin{align*}
  \Lambda_{\mathcal{T}_r} &= \left\{ \bs \alpha \in \N_0^d \;\; \big| \;\; |\bs \alpha| \leq r \right\}, &
  \Lambda_{\mathcal{H}_r} &= \left\{ \bs \alpha \in \N_0^d \;\; \big| \;\; \prod_{j=1}^d (\alpha_j+1) \leq r+1 \right\},
\end{align*}
for some non-negative integer $r$. Both of these multi-index sets are downward closed. The total order and hyperbolic cross polynomial subspaces are defined by, respectively,
\begin{align}\label{eq:polynomial-spaces}
  \Pi_{\mathcal{T}_r} &= \mathrm{span} \left\{ \bs{x}^{\bs \alpha} \;\; \big| \;\; \bs \alpha \in \Lambda_{\mathcal{T}_r} \right\}, &
  \Pi_{\mathcal{H}_r} &= \mathrm{span} \left\{ \bs{x}^{\bs \alpha} \;\; \big| \;\; \bs \alpha \in \Lambda_{\mathcal{H}_r} \right\}.
\end{align}

The algorithm we present in this paper applies to general polynomial spaces, but our numerical examples will focus on the spaces above since they are common in large-scale computing problems.

\subsection{Univariate rules: Gauss Quadrature}

When $\Gamma \subset \R$, the optimal quadrature rule is provided by the $\omega$-Gauss quadrature rule. In one dimension, we use the shorthand $\Pi_k = \Pi_{\mathcal{T}_k}$. The first step in defining this rule is to prescribe an orthonormal basis for $\Pi_{k}$. A Gram-Schmidt argument implies that such a basis of orthonormal polynomials exists with elements $p_m(\cdot)$, where $\deg p_m = m$. All univariate orthonormal polynomial families satisfy the three-term recurrence relation,
\begin{align}\label{eq:three-term-recurrence}
  x p_m(x) = \sqrt{b_m} p_{m-1}(x) + a_m p_m(x) + \sqrt{b_{m+1}} p_{m+1}(x),
\end{align}
for $m \geq 0$, with $p_{-1} \equiv 0$ and $p_0 \equiv 1/\sqrt{b_0}$ to seed the recurrence. The recurrence coefficients are given by
\begin{align*}
  a_m &= (x p_m, p_m), & b_m &= \frac{(p_m, p_m)}{(p_{m-1}, p_{m-1})},
\end{align*}
for $m \geq 0$, with $b_0 = (p_0, p_0)$. Classical orthogonal polynomial families, such as the Legendre and Hermite polynomials, fit this mold with explicit formula for the $a_n$ and $b_n$ coefficients \cite{szego_orthogonal_1975}. Gaussian quadrature rules are $n$-point rules that exactly integrate polynomials in $\Pi_{2n-1}$\cite{Stoer2002,Davis07}.

\begin{theorem}[Gaussian quadrature]\label{TH2_1}
Let $x_{1},\ldots,x_{n}$ be the roots of the $n$th orthogonal polynomial $p_n(x)$ and let $w_{1},\ldots,w_{n}$ be the solution of the system of equations
\begin{equation}\label{THM1_0}
\sum_{q=1}^n p_j(x_q) w_q =
\begin{cases}
  \sqrt{b_0}, & \textrm{if } j=0\\
  0, & \textrm{for } j=1,\ldots,n-1.\\
\end{cases}
\end{equation}
Then $x_q \in \Gamma$ and $w_q>0$ for $q=1,2,\ldots,n$ and
\begin{equation}\label{THM1_1}
\displaystyle \int_{\Gamma} \omega(x) p(x) dx = \sum_{q=1}^n p(x_q) w_q
\end{equation}
holds for all polynomials $p \in \Pi_{2n-1}$.
\end{theorem}
Historically significant algorithmic strategies for computing Gauss quadrature rules are given in \cite{Gautschi68,Golub69}. The elegant linear algebraic formulations described in these references compute the quadrature rule with knowledge of only of a finite number of recurrence coefficients $a_n$, $b_n$.

\subsection{Multivariate polynomials}\label{sec:multivariate-polynomials}
If $\Gamma$ and $\omega(\bs{x})$ are both tensorial, then the generalization of univariate orthogonal polynomials to multivariate ones is straightforward. The tensorial structure implies
\begin{align*}
  \Gamma &= \times_{j=1}^d \Gamma_j, & \omega(\bs{x}) &= \prod_{j=1}^d \omega_j\left(x^{(j)}\right),
\end{align*}
for univariate domains $\Gamma_j \subset \R$ and univariate weights $\omega_j(\cdot)$. If $p^{(j)}_n(\cdot)$ is the univariate orthonormal polynomial family associated with $\omega_j$ over $\Gamma_j$, then
\begin{align}\label{eq:tensorial-pi}
  \pi_{\bs{\alpha}}(\bs{x}) &= \prod_{j=1}^d p^{(j)}_{\alpha_j}\left(x^{(j)}\right), & \bs{\alpha} \in \N_0^d,
\end{align}
defines a family of multivariate polynomials orthonormal under $\omega$, i.e., $\left(\pi_{\bs{\alpha}}, \pi_{\bs{\beta}} \right) = \delta_{\bs{\alpha}, \bs{\beta}}$, where $\delta$ is the Kronecker delta. The polynomial spaces in \eqref{eq:polynomial-spaces} can be written as
\begin{align*}
  \Pi_{\mathcal{T}_r} &= \mathrm{span} \left\{ \pi_{\bs{\alpha}} \;\; \big| \;\; \bs{\alpha} \in \Lambda_{\mathcal{T}_r} \right\} &
  \Pi_{\mathcal{H}_r} &= \mathrm{span} \left\{ \pi_{\bs{\alpha}} \;\; \big| \;\; \bs{\alpha} \in \Lambda_{\mathcal{H}_r} \right\}
\end{align*}
The following result is the cornerstone of our algorithm:

\begin{proposition}\label{PR2_1}
  Let $\Lambda$ be a multi-index set with $\bs{0} \in \Lambda$. Suppose that $\bm x_{1},\ldots,\bm x_{n}$ and $w_{1},\ldots,w_{n}$ are the solution of the system of equations
\begin{equation}\label{PRO2_0}
\sum_{q=1}^n \pi_{\bm \alpha}(\bm x_{q}) w_{q} =
\begin{cases}
  1/\pi_{\bs{0}}, &\textrm{if } \bm \alpha= \bm 0\\
  0, &\textrm{if } \bm \alpha \in \Lambda\backslash\{\bm 0\}\\
\end{cases}
\end{equation}
then
\begin{equation}\label{PRO2_1}
\displaystyle \int_{\bm \Gamma} \omega(\bm x) \pi(\bm x) d\bm x = \sum_{q=1}^n \pi(\bm x_{q}) w_{q}
\end{equation}
holds for all polynomials $\pi \in \mathrm{\Pi}_{\Lambda}$.
\end{proposition}
The proof is straightforward by noting that $\int_\Gamma \pi_{\bs{\alpha}}(\bs{x}) \omega(\bs{x}) \dx{\bs{x}} = 0$ when $\bs{\alpha} \neq \bs{0}$ due to orthogonality, and thus \eqref{PRO2_0} is a moment-matching condition. Unlike Theorem~\ref{TH2_1}, this multivariate result does not guarantee the positivity of weights nor does it ensure that the nodes lie in $\Gamma$. We enforce these conditions in our computational framework in Section~\ref{S3}. Finally, we note that Proposition \ref{PR2_1} is true even when $\Gamma$ and $\omega$ are not tensorial. We concentrate on the tensorial situation in this paper because a tensorial assumption is standard for large dimension $d$.

One of the main uses of quadrature rules is in the construction of polynomial approximation via discrete quadrature. If $f$ is a given continuous function and $\Theta$ is a given multi-index set, then
\begin{align}\label{eq:discrete-quadrature}
  f(\bs{x}) \approx f_{\Theta}(\bs{x}) &= \sum_{\alpha \in \Theta} \widehat{f}_{\bs{\alpha}} \pi_{\bs{\alpha}}(\bs{x}), &
  \widehat{f}_{\bs{\alpha}} = \sum_{q=1}^n \pi_{\bs{\alpha}}\left(\bs{x}_{q}\right) f\left(\bs{x}_{q}\right) w_{q},
\end{align}
where $\widehat{f}_{\bs{\alpha}}$ are meant to approximate the Fourier ($L^2_\omega$-projection) coefficients of $f$. Ideally, if $f \in \Pi_{\Theta}$ then $f_{\Theta} = f$, i.e., this construction reproduces polynomials in $\Pi_\Theta$. As one expects, this only happens when the quadrature rule is sufficiently accurate, as defined by the size of $\Lambda$ in \eqref{PRO2_0}.
\begin{proposition}\label{prop:quadrature-stability}
  Let $\Lambda$ be a downward-closed multi-index set, and suppose that $\bs{x}_q$ and $w_q$ for $q = 1, \ldots, n$ define a quadrature rule satisfying \eqref{PRO2_0}. Let $\Theta$ be any index set satisfying
  \begin{align}\label{eq:Theta-assumption}
    \Theta + \Theta = \left\{ \bs{\alpha} + \bs{\beta} \;\; \big| \;\; \bs{\alpha},\bs{\beta} \in \Theta \right\} \subseteq \Lambda.
  \end{align}
  If $f \in \Pi_\Theta$, then $f_\Theta$ defined in \eqref{eq:discrete-quadrature} satisfies $f_\Theta = f$.
\end{proposition}
\begin{proof}
  Suppose $f \in \Pi_\Theta$, so that
  \begin{align*}
    f(\bs{x}) &= \sum_{\bs{\alpha} \in \Theta} f_{\bs{\alpha}} \pi_{\bs{\alpha}}(\bs{x}), & f_{\bs{\alpha}} &= \left( f, \pi_{\bs{\alpha}} \right),
  \end{align*}
  where the formula for the coefficients $f_{\bs{\alpha}}$ is due to orthogonality. We will show that the computed quadrature coefficients $\widehat{f}_{\bs{\alpha}}$ defined in \eqref{eq:discrete-quadrature} satisfy $\widehat{f}_{\bs{\alpha}} = f_{\bs{\alpha}}$. Fix $\bs{\beta} \in \Theta$. Then,
  \begin{align*}
    f(\bs{x}) \pi_{\bs{\beta}}(\bs{x}) = \sum_{\bs{\alpha} \in \Theta} f_{\bs{\alpha}} \pi_{\bs{\alpha}}(\bs{x}) \pi_{\bs{\beta}}(\bs{x}).
  \end{align*}
  There are coefficients $c_{\bs{\alpha},\bs{\gamma}}$ such that
  \begin{align*}
    \pi_{\bs{\alpha}} = \sum_{\bs{\gamma} \leq \bs{\alpha}} c_{\bs{\alpha},\bs{\gamma}} \bs{x}^{\bs{\gamma}}.
  \end{align*}
  Therefore,
  \begin{align*}
    \pi_{\bs{\alpha}}(\bs{x}) \pi_{\bs{\beta}}(\bs{x}) = \left( \sum_{\bs{\gamma} \leq \bs{\alpha}} c_{\bs{\alpha}, \bs{\gamma}} \bs{x}^{\bs{\gamma}}  \right) \left( \sum_{\bs{\gamma} \leq \bs{\beta}} c_{\bs{\beta}, \bs{\gamma}} \bs{x}^{\bs{\gamma}}  \right) = \sum_{\bs{\gamma} \leq \bs{\alpha} + \bs{\beta}} d_{\bs{\alpha}, \bs{\beta}, \bs{\gamma}} \bs{x}^{\bs{\gamma}}, 
  \end{align*}
  for some coefficients $d_{\bs{\alpha}, \bs{\beta},\bs{\gamma}}$. The index $\bs{\alpha} + \bs{\beta} \in \Lambda$ owing to the assumption \eqref{eq:Theta-assumption}, and since $\Lambda$ is downward closed, then we have that $\pi_{\bs{\alpha}}(\bs{x}) \pi_{\bs{\beta}}(\bs{x}) \in \Pi_{\Lambda}$. Therefore, the $n$-point quadrature rule integrates $\pi_{\bs{\alpha}}(\bs{x}) \pi_{\bs{\beta}}(\bs{x})$, and thus
  \begin{align*}
    \widehat{f}_{\bs{\beta}} = \sum_{q=1}^n f(\bs{x}_q) \pi_{\bs{\beta}}(\bs{x}_q) = \sum_{\bs{\alpha} \in \Theta} f_{\bs{\alpha}} \sum_{q=1}^n \pi_{\bs{\alpha}}(\bs{x}) \pi_{\bs{\beta}}(\bs{x}) = \sum_{\bs{\alpha} \in \Theta} f_{\bs{\alpha}} \left( \pi_{\bs{\alpha}}, \pi_{\bs{\beta}}\right) = f_{\bs{\beta}},
  \end{align*}
  Since $\widehat{f}_{\bs{\beta}} = f_{\bs{\beta}}$, then $f_\Theta = f$.
\end{proof}
The notion above of reproduction of multivariate polynomials is consistent with univariate Gauss quadrature: In one dimension with an $n$-point Gauss quadrature rule, we can reproduce polynomials up to degree $n-1$: Take $\Lambda = \left\{ 0, \ldots, 2n -1 \right\}$, and choose $\Theta = \left\{0, \ldots, n-1\right\}$. The polynomial $f_{\Theta}$ constructed by the procedure \eqref{eq:discrete-quadrature} matches the function $f$ if $f \in \Pi_\Theta$ since $\Theta + \Theta \subset \Lambda$. \annote{The above result codifies this condition in the multivariate case. Note that $\Theta \subset \Lambda$ is not a strict enough condition since the approximate Fourier coefficients defined in \eqref{eq:discrete-quadrature} will not necessarily be accurate. We also note that the integrand is a product of polynomials, therefore requiring exactness on polynomial products is the correct condition, hence the $\Theta + \Theta \subset \Gamma$ requirement.}

Given a multi-index set $\Lambda$, there is a smallest possible quadrature size $n$ such that \eqref{PRO2_0} holds. This smallest $n$ is given by the size of the largest $\Theta$ satisfying \eqref{eq:Theta-assumption}.
\begin{theorem}[\cite{Jakeman17quadrature}]\label{thm:half-set}
  Let $\Lambda$ be a downward-closed index set. The size $n$ of any quadrature rule satisfying \eqref{PRO2_0} has lower bound
  \begin{align*}
    n \geq \mathcal{L}(\Lambda) \coloneqq \max \left\{ \left|\Theta\right|\;\; \big| \;\; \Theta + \Theta \subseteq \Lambda \right\}.
  \end{align*}
\end{theorem}
The number $\mathcal{L}(\Lambda)$ defined above is called the maximal half-set size in \cite{Jakeman17quadrature}, and a corresponding $\mathcal{L}(\Lambda)$-point quadrature rule is a minimal rule. In that reference, concrete examples of (i) non-existence, and of (ii) existence but non-uniqueness of minimal multivariate quadrature rules achieving the lower bound above are shown. If $\Lambda = \Lambda_{2n-2}$ in the univariate case, Gaussian quadrature rules are non-unique. Our numerical algorithm essentially seeks to find minimal rules, but we can rarely find such quadrature rules. However, our generated quadrature rule sizes are only modestly larger than the optimal $\mathcal{L}(\Lambda)$.

\subsection{Quadrature Stability}\label{S2_3}

Gaussian quadrature rules defined by Theorem \ref{TH2_1} can be computed via linear algebra, but multivariate quadrature rules defined by \eqref{PRO2_0} have no known analogous computational simplification. In order to solve this nonlinear system of equations we utilize Newton's method. We therefore expect that \eqref{PRO2_0} is not exactly satisfied by the computed solution, or it is satisfied to within some tolerance.

Fixing a downward-closed index set $\Lambda$ with size $M = |\Lambda|$, consider the matrix $\bs{X} \in \R^{d \times n}$ whose $n$ columns are the samples $\bs{x}_j$, and let $\bs{w} \in \R^n$ be a vector containing the $n$ weights. Let $\bs{V}(\bs{X}) \in \R^{n \times M}$ denote the Vandermonde-like matrix with entries
\begin{align}\label{eq:Vandermonde-def}
  \left(V \right)_{k,j} &= \pi_{\bs{\alpha}(k)}\left(\bs{x}_j\right), & j = 1, \ldots, n, &\; k=1, \ldots, M,
\end{align}
where we have introduced an ordering $\bs{\alpha}(1), \ldots \bs{\alpha}(m)$ on the elements of $\Lambda$. We assume $\bs{\alpha}(1) = \bs{0}$, but the remaining ordering of elements is irrelevant. The system \eqref{PRO2_0} can then be written as
\begin{align*}
  \bs{V}\left(\bs{X}\right) \bs{w} = \bs{e}_1/\pi_{\bs{0}},
\end{align*}
where $\bs{e}_1 = (1, 0, 0, \ldots, 0)^T \in \R^M$ is a cardinal unit vector. Instead of achieving the equality above, our computational solver computes an approximate solution $\left(\bs{X},\bs{w}\right)$ to the above system, satisfying
\begin{equation}\label{compute_GQ2_0}
  \left\| \bs{V}\left(\bs{X}\right) \bs{w} - \bs{e}_1/\pi_{\bs{0}} \right\|_2 = \epsilon \geq 0.
\end{equation}
Our next result quantifies the effect of the residual $\epsilon$ on the accuracy of the designed quadrature rule. \annote{To prove this result, we require the additional assumption that the quadrature weights are positive, which is enforced in our computations.}

\begin{proposition}\label{PR2_2}
  Let $\omega(\bs{x})$ be a probability density function on $\Gamma$, and let $\Lambda$ be any multi-index set containing $\bs{0}$ (i.e., $\Pi_\Lambda$ contains constant functions). Assume that $(\bs{X}, \bs{w})$ satisfies
  \eqref{compute_GQ2_0} with some $\epsilon \geq 0$, and assume the weights are all positive. Then for any $f \in L^2_\omega(\Gamma)$,
  \begin{align}\label{eq:quadrature-estimate}
    \left| \int f(\bs{x}) \omega(\bs{x}) \dx{\bs{x}} - \sum_{q=1}^n w_q f(\bs{x}_q) \right| \leq \epsilon \left\| f\right\| + \max_{j=1,\ldots n} \left| f(\bs{x}_j) - p(\bs{x}_j) \right|,
  \end{align}
  where $p \in \Pi_\Lambda$ is the $L^2_\omega(\Gamma)$-orthogonal projection of $f$ onto $\Pi_\Lambda$.
\end{proposition}
This result does apply to all our computed designed quadrature rules since we enforce positivity of the weights. \annote{It is not applicable to other polynomial-based rules where weights can be negative, such as sparse grids.}
\begin{proof}
  For an arbitrary $p \in \Pi_\Lambda$, the following holds
  \begin{align}\label{eq:p-expansion}
    p(\bs{x}) &= \sum_{\bs{\alpha} \in \mathcal{I}} p_{\bs{\alpha}} \pi_{\bs{\alpha}}(\bs{x}), & p_{\bs{\alpha}} &= \left(p, \pi_{\bs{\alpha}}\right),
  \end{align}
  and thus $\|p\|^2 = (p,p) = \sum_{\bs{\alpha} \in \Lambda} p_{\bs{\alpha}}^2$.
  We have:
  \begin{align}\label{eq:quadbound-decomposition}
    \left| \int_{\Gamma} f(\bs{x}) \omega(\bs{x}) \dx{\bs{x}} - \sum_{q=1}^n w_q f\left(\bs{x}_q\right) \right| &\leq \underbrace{\left| \int_{\Gamma} ( f(\bs{x}) - p(\bs{x}) ) \omega(\bs{x}) \dx{\bs{x}}\right|}_{(a)} + \underbrace{\left| \sum_{q=1}^n \left(p(\bs{x}_q)
    - f(\bs{x}_q)\right) w_q \right|}_{(b)} \\\nonumber
    &+\underbrace{\left| \int_\Gamma p(\bs{x}) \omega(\bs{x}) \dx{\bs{x}} - \sum_{q=1}^n w_q p(\bs{x}_q) \right|}_{\textrm{(c)}}
  \end{align}
  We now choose $p$ as the $L^2_\omega(\Gamma)$-orthogonal projection of $f$ into $\Pi_\Lambda$:
  \begin{align}\label{eq:f-orthogonal-projection}
    p = \argmin_{q \in \Pi_\Lambda} \left\| f - q \right\| \;\; \Longrightarrow \int_\Gamma \left[ f(\bs x) - p(\bs x) \right] \phi(\bs x) \omega(\bs x) \dx{\bs x} = 0 \;\;\; \forall \; \phi \in \Pi_\Lambda.
  \end{align}
  Since $\bs{0} \in \Lambda$, the above holds in particular for $\phi(\bs x) \equiv 1$ so that
  \begin{align*}
    \textrm{(a)} &=  \left| \int_{\Gamma}  \left(f(\bs{x}) - p(\bs{x})\right)  \omega(\bs{x}) \dx{\bs{x}} \right| = 0
  \end{align*}
  Term (b) can be bounded as
  \begin{align*}
    \textrm{(b)} &\leq \sum_{q=1}^n |w_q| \left| p(\bs{x}_q) - f(\bs{x}_q) \right| \leq \max_{q=1,\ldots, n} \left| p(\bs{x}_q) - f(\bs{x}_q) \right|,
  \end{align*}
  where the last inequality uses the fact that $\sum_{q=1}^N |w_q| = \sum_{q=1}^N w_q  = \int_\Gamma \omega(\bs{x}) \dx{\bs{x}} = 1$ since the weights are positive and $\omega$ is a probability density. Finally, term (c) can be bounded as follows: Since $p \in \Pi_\Lambda$ then by \eqref{eq:p-expansion},
  \begin{align*}
  \sum_{q=1}^n w_q p(\bs{x}_q) = \sum_{q=1}^n \sum_{\bs{\alpha} \in \Lambda} w_q p_{\bs{\alpha}} \pi_{\bs{\alpha}}(\bs{x}_q) = \sum_{\bs{\alpha} \in \Lambda} p_{\bs{\alpha}} \left( \sum_{q = 1}^n w_q \pi_{\bs{\alpha}}(\bs{x}_q)\right)
  \end{align*}
  The term in parenthesis on the right-hand side is an entry in the vector $\bs{V}(\bs{X}) \bs{w}$ from the relation \eqref{compute_GQ2_0}; note also that $\widehat{\pi}_{\bs{\alpha}}$ cf. Equation~\eqref{eq:discrete-quadrature} equals an entry in the vector $\bs{b}$. Therefore, combining the above equation and using the Cauchy-Schwarz inequality:
  \begin{align*}
    \textrm{(c)} &= \left| \int_\Gamma p(\bs{x}) \omega(\bs{x}) \dx{\bs{x}} - \sum_{q=1}^n w_q p(\bs{x}_q) \right| = \left| \sum_{\bs{\alpha} \in \Lambda} p_{\bs{\alpha}} \left( \int_\Gamma \pi_{\bs{\alpha}}(\bs{x}) \omega(\bs{x}) \dx{\bs{x}} -  \sum_{q=1}^n w_q \pi_{\bs{\alpha}}(\bs{x}_q) \right) \right| \\
                 &= \left| \sum_{\bs{\alpha} \in \Lambda} p_{\bs{\alpha}} \left( \delta_{\bs{\alpha},\bs{0}}/\pi_{\bs{0}} -  \sum_{q=1}^n w_q \pi_{\bs{\alpha}}(\bs{x}_q) \right) \right|
    \leq \sqrt{\sum_{\bs{\alpha} \in \Lambda} p_{\bs{\alpha}}^2} \left\| \bs{V}(\bs{X}) \bs{w} - \bs{e}_1/\pi_{\bs{0}} \right\| \leq \epsilon \left\| p \right\| \leq \epsilon \|f \|,
  \end{align*}
  where the final inequality is Bessel's inequality, which holds since we have chosen $p$ as in \eqref{eq:f-orthogonal-projection}. Combining our estimates for terms (a), (b), and (c) in \eqref{eq:quadbound-decomposition} completes the proof.
\end{proof}

Relative to \annote{the pointwise error committed by} best $L^2_\omega(\Gamma)$ approximations, the estimate provided by Proposition \ref{PR2_2} bounds the quadrature error in terms of the quantity $\epsilon$, which is explicitly computable given a quadrature rule.

\subsection{A popular alternative: Sparse Grids}

A (Smolyak) sparse grid is a structured point configuration in multiple dimensions, formed from unions of tensorized univariate rules. Quadrature weights often accompany points in a sparse grid. We briefly describe sparse grids for polynomial integration in this section; they will be used for comparison in our numerical results section.

Consider a tensorial $\Gamma$ as in Section \ref{sec:multivariate-polynomials}, and for simplicity assume that the univariate domains $\Gamma_j = \Gamma_1$ are the same, and that the univariate weights $\omega_j = \omega_1$ are the same. Let $\mathbb{X}_{i}$ denote a univariate quadrature rule (nodes and weights) of ``level" $i \geq 1$, and define $\mathbb{X}_0 = \emptyset$. The number of points $n_i$ in the quadrature rule $\mathbb{X}_i$ is increasing with $i$, but can be freely chosen. For multi-index $\bs{i} \in \N^d$, a $d$-variate tensorial rule and its corresponding weights are
\begin{equation}
\label{SP1} \displaystyle \mathbb{A}_{d,\bm i}= \mathbb{X}_{i_1} \otimes \ldots \otimes \mathbb{X}_{i_d}, \quad \displaystyle w^{(\bm q)}= \prod_{r=1}^d w_{i_r}^{(q_r)}
\end{equation}
The univariate difference operator between sequential levels is written as
\begin{align}\label{SP2}
  \Delta_i & = \mathbb{X}_{i} - \mathbb{X}_{i-1}, & i &\geq 1,
\end{align}
and for any $k \in \N$, this approximation difference can be used to construct a $d$-variate, level-$k$-accurate sparse grid operator \cite{Bungartz04,Smol63},
\begin{align}\label{SP3}
  \mathbb{A}_{d,k} = \sum_{r=0}^{k-1} \sum_{\substack{\bs{i} \in \N^d \\\left| \bs{i} \right| = d+r}} \Delta_{i_1} \otimes \ldots \otimes \Delta_{i_d}
   = \sum_{r=k-d}^{k-1} (-1)^{k-1-r}  \binom{d-1}{k-1-r} \displaystyle \sum_{\substack{\bs{i} \in \N^d \\\left| \bs{i} \right| = d+r}} \mathbb{X}_{i_1} \otimes \ldots
\otimes \mathbb{X}_{i_d},
\end{align}
where the latter equality is shown in \cite{Wasilkowski95}. If the univariate quadrature rule $\mathbb{X}_i$
exactly integrates univariate polynomials of order $2i-1$ or less, then the Smolyak rule $\mathbb{A}_{d,k}$ is exact for $d$-variate polynomials of total order $2k-1$ \cite{Heiss08}. One is tempted to use Gauss quadrature rules for the $\mathbb{X}_i$ to obtain optimal efficiency, but since the differences $\Delta_i$ appear in the Smolyak construction, then instead utilizing nested univariate rules can generate sparse grids with many fewer nodes than non-nested constructions. One can use, for example, nested Clenshaw-Curtis rules \cite{xiu_high-order_2005}, the nested Gauss-Patterson or Gauss-Kronrod rules \cite{gerstner_numerical_1998,liu_adaptive_2011,Patterson68}, or Leja sequences \cite{narayan_adaptive_2014}.

Sparse grids have been used with great success in many modern applications, and thus are a good candidate for comparison against our approach of designed quadrature. However, sparse grids that integrate polynomials in a certain multi-index set use far more points than the minimum number prescribed by Theorem \ref{thm:half-set} \annote{(see Figure \ref{fig_2_0} for an empirical comparison)}, and frequently produce quadrature rules with negative weights. Our results in Section \ref{sec:results} show that designed quadrature uses many fewer points than sparse grids for a given accuracy level, and guarantees positive quadrature weights.

\section{Computational Framework}\label{S3}

Our procedure aims to compute nodes $\bs{X} = \left\{ \bs{x}_1, \ldots, \bs{x}_n \right\}\in \Gamma^n$
and positive weights $\bs{w} \in (0, \infty)^n$ that enforce equality in \eqref{PRO2_0}. A direct formulation of \eqref{PRO2_0} is

\begin{equation}\label{S3_1_0}
\begin{array}{r l}
  \bm R(\bm d) = \bm V(\bm X) \bm w - \bs{e}_1/\pi_{\bs{0}} = \bm 0, & \\
  \bm x_j \in \Gamma, & j = 1, \ldots, n\\
  w_j > \bm 0, & j = 1, \ldots n
\end{array}
\end{equation}
where $\bm d = (\bm X, \bm w)$ are the decision variables. Instead of directly solving this constrained root finding problem, we introduce a closely related constrained optimization problem
\begin{equation}\label{S3_3_3}
\begin{array}{r l l}
  \displaystyle \mathop{\min}_{\bm X, \bm w} & \displaystyle ||\bm R||_2  &  \\
  \text{subject to} & \bm x_j \in \Gamma, & j=1, \ldots, n \\
                    & w_j> \bm 0, & j =1, \ldots, n
\end{array}
\end{equation}
Clearly a solution to \eqref{S3_1_0} also solves \eqref{S3_3_3}, but the reverse is not necessarily true. \annote{We compute solutions to \eqref{S3_3_3}, and when these solutions exhibit large nonzero values of $\|\bs{R}\|$, we increase the quadrature rule size $n$ and repeat. Using this strategy, we empirically find that for a specified $\epsilon$ we can satisfy $\|\bs{R}\| \leq \epsilon$ in all situations we have tried. Thus, our approach solves a relaxed version of \eqref{S3_1_0} via repeated applications of \eqref{S3_3_3}.} Our computational approach to solve \eqref{S3_3_3} requires four major ingredients, each of which are described in the subsequent sections:
\begin{enumerate}[leftmargin=1in]
  \item[Section \ref{S3_3} --] Penalization: objective augmentation, transforming constrained root finding into unconstrained minimization problem
  \item[Section \ref{sec:method-newton} --] Iteration: unconstrained minimization via the Gauss-Newton algorithm
  \item[Section \ref{S3_4} --] Regularization: numerical regularization to address ill-conditioned Gauss-Newton update steps
  \item[Section \ref{S3_2} --] Initialization: specification of an initial guess
\end{enumerate}
\annote{We highlight above that regularization is required for our optimization. The objective $\bs{R}$ in \eqref{S3_3_3} is highly ill-conditioned as a function of the decision variables. Without regularization, the update steps specified by the Gauss-Newton algorithm generally do not result in convergence. However, with the regularization, we have found that our optimization results in steps with decreasing residual. These observations can be corroborated by the numerical results in Section \ref{sec:results}, and in particular Table~\ref{tab:iterations} that lists CPU time and iterations required for computing 4-dimensional rules.

  Since our algorithm only minimizes the norm of $\bs{R}$, the quadrature rule we compute is not guaranteed to integrate any polynomials exactly, only up to some tolerance parameter $\epsilon \geq \|\bs{R}\|$. This is the utility of Proposition \ref{PR2_2}: if our optimization algorithm terminates with a particular value of $\epsilon$, we have a quantitative understanding of how $\epsilon$ affects the quality of the quadrature rule relative to best $L^2$-approximating polynomials.

  Since we produce a quadrature rule that is only $\epsilon$-exact, there may be many quadrature rules that achieve this tolerance. In particular, our algorithm is not guaranteed to produce optimal quadrature rules, but in comparison with some other tabulated rules from \cite{Stroud67,stroud_numerical_1960,Xiao10,xiu_numerical_2008}, we find that our nodal counts are no greater than in those references. There is one lone exception for integrating degree-8 polynomials in three dimensions, where we find a rule with one point greater than reported in \cite{Xiao10}. Details are in Section \ref{sec:results-comparison} and in Table \ref{tab:comparison}.

  Finally, our algorithm is subject to the same limitations as many other minimization algorithms: it may only find a local minimum of the objective, and not a global minimum.
}

\subsection{Penalization}\label{S3_3}

Penalty methods are techniques to solve constrained optimization problems such as~\eqref{S3_3_3}. Penalty methods augment the objective with a high cost for constraint violated, and subsequently solve an unconstrained optimization problem on the augmented objective.

We use a popular penalty function, the non-negative and smooth quadratic function. For example in $d=1$ dimensions on $\Gamma = [-1,1]$ with an $n$-point quadrature rule, the constraints and corresponding penalties $P_j$, $j=1, \ldots, (d+1)n = 2n$ as a function of the $2n$ decision variables $\bs{d} = \left(\bs{X}, \bs{w}\right)$ can be expressed as
\begin{align*}
  -1 \leq x_j \leq 1 \hskip 5pt &\Longrightarrow \hskip 5pt P_j\left(\bs{d}\right) = \left(\max[0,x_j-1,-1-x_j]\right)^2,\\
  w_j \geq 0 \hskip 5pt &\Longrightarrow \hskip 5pt P_{n+j}\left(\bs{d}\right) = \left(\max[0,-w_j]\right)^2,
\end{align*}
for $j = 1, \ldots, n$. The total penalty associated with the constraints is then
\begin{align*}
  P^2\left(\bs{d}\right) = \sum_{j=1}^{(d+1)n} P^2_j\left(\bs{d}\right).
\end{align*}

A penalty function approach to solve the constrained problem \eqref{S3_3_3} uses a sequence of unconstrained problems indexed by $k \in \N$ having objective functions
\begin{align}\label{eq:g-def}
  g\left(c_k, \bs{d} \right) \coloneqq \left\|\widetilde{\bs{R}}_k \right\|^2_2 = \left\| \bs{R}\right\|^2_2 + c^2_k P^2\left(\bs{d}\right),
\end{align}
where we have defined the vector
\begin{align*}
  \widetilde{\bs{R}}_k &= \left[
    \begin{array}{c}
      \bs{R} \\
      c_k P_1 \\
      c_k P_2 \\
      \vdots \\
      c_k P_{(d+1)n}
    \end{array}
  \right].
\end{align*}
The positive constants $c_k$ are monotonically increasing with $k$, i.e., $c_{k+1}> c_k$. Each unconstrained optimization yields an updated solution point $\bm d^{k}$, and as $c_k \rightarrow \infty$ the solution point of the unconstrained problem will converge to the solution of constrained problem. The following lemma, adopted from~\cite{Luenberger08}, is used to show convergence of the penalty method.
\begin{lemma}\label{TH3_3_1}
Let $\bm d^{k}$ be the minimizer for $g(c_k, \cdot)$ and $c_{k+1}> c_k$. Then:
\begin{align*}
  g(c_k, \bm d^k) &\leq g(c_{k+1}, \bm d^{k+1}), &
  P(\bm d^{k}) &\geq P(\bm d^{k+1}), &
  ||\bm R(\bm d^{k})|| &\leq ||\bm R(\bm d^{k+1})||.
\end{align*}
Furthermore, let $\bm d^{*}$ be a solution to problem~\eqref{S3_3_3}. Then for each $k$,
\begin{equation*}
||\bm R(\bm d^k)|| \leq g(c_k, \bm d^k) \leq ||\bm R(\bm d^*)||.
\end{equation*}
\end{lemma}

The above lemma denotes that the sequence of $g(c_k,\bm d^{k})$ is nondecreasing and bounded above by the optimal objective value of the constrained optimization problem. The following theorem establishes the global convergence of the penalty method. More precisely
it verifies that any limit point of the sequence is a solution to $\eqref{S3_3_3}$.
\begin{theorem}[\cite{Luenberger08}]\label{TH3_3_3}
  Let $\{\bm d^k\},~k \in \N$ be a sequence of minimizers of \eqref{eq:g-def}. Then any limit point of the sequence is a solution to problem~\eqref{S3_3_3} i.e. $\lim_{k \in \N} P(\bm d^{k})=0$ and $ \lim_{k \in \N} ||R(\bm d^k)|| \leq ||R(\bm d^*)||$.
\end{theorem}
The above theorem shows both that a limit point denoted by $\bar{\bm d}$ is a feasible solution since $P(\bar{\bm d})=0$, and that it is optimal since $||\bm R(\bar{\bm d})||^2_2 \leq ||\bm R(\bm d^*)||^2_2$.

We can now formulate an unconstrained minimization problem with sequence of increasing $c_k$ on the objectives $g$ in \eqref{eq:g-def} for the decision variables $\bs{d} = (\bs{X}, \bs{w})$,
\begin{align}\label{eq:unconstrained-g}
  \displaystyle \mathop{\min}_{\bs{d}} \displaystyle g(c_k, \bs{d})
\end{align}
which replace the constrained root-finding problem \eqref{S3_1_0}.

It remains for us to specify how the constants $c_k$ are chosen: if $\bs{d}$ is the current iterate for the decision variables, we use the formula
\begin{align*}
  c_k = \max \left\{A, \frac{1}{||\bs{R}(\bm d)||_2} \right\},
\end{align*}
where $A$ is a tunable parameter that is meant to be large. We use $A = 10^3$ in our simulations. Also note that we never have $c_k = \infty$ so that our iterations cannot exactly constrain the computed solution to lie in the feasible set. To address this in practice we reformulate constraints to have non-zero penalty within a small radius inside the feasible set. For example, instead of enforcing $w_j > 0$, we enforce $w_j > 10^{-6}$.

Note that one may also consider barrier/interior point methods to enforce constraints; however, in our algorithm we find that penalty methods are more suitable in transforming the constrained root finding problem to an unconstrained minimization problem.

\subsection{The Gauss-Newton algorithm}\label{sec:method-newton}
Having transformed the constrained problem \eqref{S3_3_3} into a sequence of unconstrained problems \eqref{eq:unconstrained-g}, we can now use standard unconstrained optimization tools.

Two popular approaches for unconstrained optimization are gradient descent and Newton's method. Both approaches in the context of our minimization require the Jacobian of the objective function with respect to the decision variables. We define

\begin{align}\label{S3_3_6}
  \widetilde{\bs{J}}_k = \pfpx{\widetilde{\bs{R}}_k}{\bs{d}} &= \left[
    \begin{array}{c}
      \bs{J} \\ 
      c_k \partial P_1 /\partial \bm d \\
      c_k \partial P_2 / \partial \bm d  \\
      \vdots \\
      c_k \partial P_{(d+1)n} / \partial \bm d
    \end{array}
  \right], &
  \bs{J}(\bs{d}) &\coloneqq \pfpx{\bs{R}}{\bs{d}} \in \R^{M \times (d+1)n},
\end{align}
where $\pfpx{P_j}{\bs{d}} \in \R^{1 \times (d+1)n}$ is the Jacobian of $P_j$ with respect to the decision variables. With use of our quadratic penalty function, these penalty Jacobians are Lipschitz continuous in the decision variables, and easily evaluated since they are quadratic functions. The matrix $\bs{J}$ has entries
\begin{align}\label{eq:jacobian-entries}
  (J)_{m, (i-1)d + j} &= \pfpx{\pi_{\bs{\alpha}(m)}\left(\bs{x}_i\right)}{x_i^{(j)}} w_i,  &
  (J)_{m, n d + i} &= \pi_{\bs{\alpha}(m)}\left(\bs{x}_j\right),
\end{align}
for $m = 1, \ldots, M$, $i = 1, \ldots, n$, and $j = 1, \ldots, d$. Above, we define $\pi_{\bs{\alpha}(m)}$ as in \eqref{eq:Vandermonde-def}. Computing entries of the Jacobian matrix $\bs{J}$ is straightforward: Assuming the basis $\bs{\pi}_{\bs{\alpha}}$ is of tensor-product form, (see Section \ref{sec:multivariate-polynomials}) then we need only compute derivatives of univariate polynomials. A manipulation of the three-term recurrence relation \eqref{eq:three-term-recurrence} yields the recurrence
\begin{align*}
  \sqrt{b_{m+1}} p_{m+1}'(x) = (x - a_m) p_m'(x) - \sqrt{b}_m p_{m-1}'(x) + p_m(x).
\end{align*}
The partial derivatives in $\bs{J}$ may be evaluated using the relation above along with \eqref{eq:tensorial-pi}.

We index iterations with $k$, which is the same $k$ as that defining the sequence of unconstrained problems \eqref{eq:unconstrained-g}. Thus, our choice of $c_k$ changes at each iteration. Gradient descent proceeds via iteration of the form
\begin{align*}
  \bm d^{k+1} &= \bm d^{k} - \alpha \pfpx{\|\widetilde{\bs{R}}_k\|_2}{\bs{d}}, & \pfpx{\|\widetilde{\bs{R}}_k\|_2}{\bs{d}} = \frac{\widetilde{\bs{J}}_k^T \widetilde{\bs{R}}}{\|\widetilde{\bs{R}}_k\|_2},
\end{align*}
with $\alpha$ a customizable step length that is frequently optimized via, e.g., a line-search algorithm. In contrast, a variant of Newton's root finding method applied to rectangular systems is the Gauss-Newton method \cite{Stoer2002}, having update iteration
\begin{align}\label{eq:gauss-newton}
  \bm d^{k+1} &= \bm d^{k} - \Delta \bs{d}, & \Delta \bs{d} &= \left( \widetilde{\bs{J}}_k^T \widetilde{\bs{J}}_k \right)^{-1} \widetilde{\bs{J}}_k^T \widetilde{\bs{R}}_k,
\end{align}
where both $\widetilde{\bs{J}}_k$ and $\widetilde{\bs{R}}_k$ are evaluated at $\bs{d}^k$. The iteration above reduces to the standard Newton's method when the system is square, i.e., $M = n(d+1)$. Newton's method converges quadratically to a local solution for a sufficiently close initial guess $\bs{d}^0$ versus the gradient descent which has linear convergence \cite{Bertsekas08}. We find that Gauss-Newton iterations are robust for our problem.

Assuming an initial guess $\bs{d}^0$ is given, we can repeatedly apply the Gauss-Newton iteration \eqref{eq:gauss-newton} until a stopping criterion is met. We terminate our iterations when the residual norm falls below a user-defined threshold $\epsilon$ i.e. $||\widetilde{\bm R}||_2 < \epsilon$.

A useful quantity to monitor during the iteration process is the magnitude of the Newton decrement, which often reflects quantitative proximity to the optimal point \cite{Boyd04}. In its original form, the Newton decrement is the norm of the Newton step in the quadratic norm defined by the Hessian. I.e., for optimizing $f(\bs{x})$, the Newton decrement norm is $|| \Delta \bm d||_{\nabla^2 f(\bm x)} = (\Delta \bm d^T \nabla^2 f(\bm x) \Delta \bm d  )^{1/2}$, where $\nabla^2 f$ is the Hessian of $f$. In our minimization procedure with non-squared systems we use
\begin{equation}\label{S3_1_6}
  \eta  = \big(\Delta \bm d^T (\widetilde{\bs{J}}_k^T \widetilde{\bs{R}}_k)\big)^{1/2}.
\end{equation}
as a surrogate for a Hessian-based Newton decrement which decreases as $\bm d \rightarrow \bm d^{*}$.

Finally we note that, for a given quadrature rule size $n$, we cannot guarantee that a solution to \eqref{S3_1_0} exists. In this case our Gauss-Newton iterations will exhibit residual norms stagnating at some positive value while the Newton decrement is almost zero. \annote{When this occurs, we re-initialize the decision variables and enrich the current set of decision variables with additional nodes and weights and continue the optimization procedure. This procedure of gradually increasing the number of nodes and weights is described more in Section \ref{S3_2}.}

\subsection{Regularization}\label{S3_4}
The critical part of our minimization scheme is the evaluation of Newton step \eqref{eq:gauss-newton}. For our rectangular system, this is the least-squares solution $\Delta \bs{d}$ to the linear system
\begin{align*}
  \widetilde{\bs{J}} \Delta \bs{d} = \widetilde{\bs{R}},
\end{align*}
where $\widetilde{\bs{J}} = \widetilde{\bs{J}}_k\left(\bs{d}^k\right)$, and $\widetilde{\bs{R}} = \widetilde{\bs{R}}_k\left(\bs{d}^k\right)$; in this section we omit explicit notational dependence on the iteration index $k$. The matrix $\widetilde{\bs{J}}$ is frequently ill-conditioned, which hinders a direct solve of the above least-squares problem. To address this we can consider a generic regularization of the above equality:

\begin{equation}\label{S3_4_1}
  \displaystyle \mathop{ \textrm{minimize}}_{\Delta \bm d} \quad ||\widetilde{\bs{J}} \Delta \bm d - \widetilde{\bs{R}} ||_p  \quad \textrm{subject to} \quad  || \Delta \bm d||_q < \tau,
\end{equation}
where $p$, $q$, and $\tau$ are free parameters. The trade off between the objective norm and solution norm is characterized as a Pareto curve and shown to be convex in~\cite{Vandenberg08,Vandenberg11} for generic norms $1 \leq (p,q) \leq \infty$. Exploiting this Pareto curve, the authors in~\cite{Vandenberg08,Vandenberg11} devise an efficient algorithm and implementation~\cite{Vandenberg_code} for computing the regularized solution when $p=2,q=1$. These values correspond to the LASSO problem \cite{tibshirani_regression_1996}, which promotes solution sparsity and subset selection.

Since sparsity is not our explicit goal, we opt for $p = q = 2$. This problem can be solved exactly \cite{golub_matrix_1996}, but at significant expense and the procedure lacks clear guidance on choosing $\tau$. We thus adopt an alternative approach. A penalized version of the $p=q=2$ optimization \eqref{S3_4_1} is Tikhonov regularization:
\begin{equation}\label{S3_4_2}
\Delta \bm d_{\lambda} = \textrm{argmin} \Big \{ ||\widetilde{\bs{J}} \Delta \bm d - \widetilde{\bs{R}} ||^2_2 + \lambda ||\Delta \bm d||_2^2  \Big \},
\end{equation}
where $\lambda$ is a regularization parameter that may be chosen by the user. This parameter has significant impact on the quality of the solution with respect to the original least-squares problem. Assuming that we have a definitive value for $\lambda$, then the solution to \eqref{S3_4_2} can be obtained via the singular value decomposition (SVD) of $\widetilde{\bs{J}}$. The SVD of matrix  $\widetilde{\bs{J}}_{\mathrm{N} \times \mathrm{M}}$ (for $\mathrm{N}<\mathrm{M}$) is given by
\begin{equation}
  \widetilde{\bs{J}} = \displaystyle \sum_{i=1}^{\mathrm{N}} \bm u_i \sigma_i \bm v^T_i.
\end{equation}
where $\sigma_i$ are singular values (in decreasing order), and $\bs{u}_i$ and $\bs{v}_k$ are the corresponding left- and right-singular vectors, respectively. The solution $\Delta \bm d_{\lambda}$ is then obtained as
\begin{equation}\label{S3_4_5}
  \Delta \bm d_{\lambda} = \displaystyle \sum_{i=1}^{\mathrm{N}} \rho_i   \displaystyle \frac{ \bm u^T_i \widetilde{\bs{R}}}{ \sigma_i} \bm v_i.
\end{equation}
where $\rho_i$ are Tikhonov filter factors denoted by
\begin{equation}\label{S3_4_4}
\begin{array}{l }
{\rho_i} =  \displaystyle \frac{\sigma_i^2}{\sigma_i^2+\lambda^2} \simeq \begin{cases}
1 \hspace{1.25cm} \sigma_i \gg \lambda  \\
\sigma_i^2/\lambda^2 \hspace{0.5cm}   \sigma_i \ll \lambda
\end{cases}
\end{array}
\end{equation}
Tikhonov regularization affects (or filters) singular values that are below the threshold $\lambda$. Therefore a suitable $\lambda$ is bounded by the extremal singular values of $\widetilde{\bs{J}}$. One approach to select $\lambda$ is via analysis of the ``$L$-curve" of singular values \cite{Hansen98,Hansen93}. The corner of $L$-curve can be interpreted as the point with maximum curvature; evaluation or approximation of the curvature with respect to singular value index can be used to find the index with maximum curvature, and the singular value corresponding to this index prescribes $\lambda$.

In practice, we evaluate the curvature of the singular value spectrum via finite differences on $\log(\sigma_i)$ (where the singular values are directly computed) and select the singular value that corresponds to the first spike in the spectrum. The regularization parameter can be updated after several, e.g., $30$, Gauss-Newton iterations. However, for small size problems, i.e., small dimension $d$ and $|\Lambda|$, a fixed appropriate $\lambda$ throughout the Gauss-Newton scheme also yields solutions.

Based on our numerical observations, adding a regularization parameter to all singular values and computing the regularized Newton step as $\Delta \bm d_{\lambda} = \sum_{i=1}^{\mathrm{N}}  [({ \bm u^T_i \widetilde{\bs{R}}})/({ \sigma_i + \lambda})] \bm v_i$ enhances the convergence when $\bm d $ is close to the root i.e. $||\widetilde{\bs{R}}||$ is small.

\subsection{Initialization }\label{S3_2}

The first step of the algorithm requires an initial guess $\bs{d}^0$ for nodes and weights; a particularly difficult aspect of this is the initial choice of quadrature rule size $n$. Our algorithm tests several values of quadrature rule sizes $n$ between an upper and lower bound; the determination of these bounds are described below.

With the multi-index set $\Lambda$ given, Theorem \ref{thm:half-set} provides a lower bound on the value of $n$, and this lower bound $\mathcal{L}(\Lambda)$ is the optimal size for a quadrature rule. We are unaware of sufficient conditions under which optimal quadrature rules exist. However, optimal-size quadrature rules have been shown in special cases, e.g., \cite{Stroud71}, for total degree spaces $\Lambda_{\mathcal{T}_k}$ with $k=2, 3, 5$. We have found that our algorithm is able to recover these optimal-sized rules in the previously-mentioned cases.

We formulate an upper bound on quadrature rule sizes based on a popular competitor: sparse grid constructions. The number of sparse grid points $|\mathbb{A}_{d,k}|$ required to satisfy \eqref{PRO2_0} with $\Lambda = \Lambda_{\mathcal{T}_k}$ can be estimated as $|\mathbb{A}_{d,k}| \approx \frac{(2d)^{k-1}}{(k-1)!}$ \cite{constantine_sparse_2012} for sparse grid constructions with non-nested univariate Gauss quadrature rules. Tabulation of the exact number of points for sparse grids constructed via univariate nested rules from the Hermite and Legendre systems is provided in~\cite{Qsparse}.

Our numerical results show that the number of designed quadrature nodes needed to satisfy \eqref{PRO2_0} is $n =\kappa |\mathbb{A}_{d,k}|$ where $\kappa \in [0.5,0.9]$ using $|\mathbb{A}_{d,k}|$ from~\cite{Qsparse}. We have found that an effective approach to choose the number of points is to perform a backtracking line-search procedure, which initializes $\kappa = 0.9$, solves the optimization problem, and gradually decreases $\kappa$ until the Gauss-Newton method does not converge to a desirable tolerance. Our strategy for eliminating nodes when $\kappa$ is decreased is to discard those with the smallest weights.

\annote{After the initial pass that generates $n$ nodes and weights achieving $\|\widetilde{\bs{R}}\| \leq \epsilon$, we attempt to remove nodes with smallest weights as described previously. However, this may cause the optimization to stagnate without achieving the desired tolerance. When this happens, we enrich the nodal set by gradually adding more nodes until we can achieve the tolerance. This process is repeated until the elimination and enrichment procedures result in no change of the quadrature rule size; see Algorithm \ref{alg:quad-design}, lines 9-18.}

Once an initial number of nodes $n$ is determined ($\kappa = 0.9$), that number of $d$-variate Monte Carlo samples or Latin Hypercube samples are generated as the initial nodes. This is easily done for the domain $\Gamma = [-1,1]^d$. Weights can be generated uniformly at random $[0,1]$ with $\sum_i w_i = |\Lambda|$ or set as a fixed value, e.g., $w_i=|\Lambda|/n$. We normalize the weights by $|\Lambda|$ in the numerical procedure to avoid very small weights. To accommodate for this, we can set $1/\pi_{\bs{0}} = |\Lambda|$ in \eqref{S3_1_0}, and after we obtain a solution we can re-normalize the weights based on the true value of $1/\pi_{\bs{0}}$.

On the domain $\Gamma = \R^d$, we are usually concerned with the weight $\omega(\bs{x}) = \exp(-\|\bs{x}\|_2^2)$. Monte Carlo samples can be generated as realizations of a standard normal random variable, and we transform Latin Hypercube samples on $[0,1]^d$ to $\R^d$ via inverse transform sampling corresponding to a standard normal random variable. (When $\Lambda$ contains polynomials of very high degree there are more sophisticated sampling methods that can produce better initial guesses \cite{Narayan17}.) We initialize the weights by setting $w_i=\exp(-||x_i||^2_2/2)$ and normalizing $w_i$ with respect to $|\Lambda|$ as described above.

Algorithm \ref{alg:quad-design} summarizes Sections \ref{S3_3}--\ref{S3_2}, including all the steps for our designed quadrature method.

\begin{algorithm}
\caption{Designed Quadrature}\label{alg:quad-design}
\label{alg1}
\begin{algorithmic}[1]
\STATE Initialize nodes and weights $\bs{d}$ with $n=0.9|\mathbb{A}_{d,k}|$ and specify the residual tolerance, e.g., $\epsilon=10^{-8}$.
\STATE Set $n_0 = 0$.
\WHILE{ $||\widetilde{\bs{R}}||> \epsilon$}
\STATE Compute $\widetilde{\bs{R}}$ and $\widetilde{\bs{J}}$ using \eqref{eq:g-def}, \eqref{S3_1_0}, \eqref{S3_3_6}, and \eqref{eq:jacobian-entries}.
\STATE Determine the regularization parameter $\lambda$ from the SVD of $\widetilde{\bs{J}}$
\STATE Compute the regularized Newton step $\Delta \bm d$ from \eqref{S3_4_5}
\STATE Update the decision variables $\bm d^{k+1} = \bm d^{k} - \Delta \bm d$.
\STATE Compute the residual norm $||\widetilde{\bs{R}}||_2$ and Newton decrement $\eta$ from \eqref{S3_1_6}.
\IF {$\eta < \epsilon$ and $||\widetilde{\bs{R}}||_2 \gg \epsilon$}
\STATE Increase $n$, initialize new nodes and weights, and go to line $3$.
\ENDIF
\ENDWHILE
\IF {$n = n_0$}
\STATE Return
\ELSE
\STATE $n_0 \gets n$
\STATE Decrease $n$ by eliminating nodes with smallest weights, go to line $3$. \annote{(See discussion about $\kappa$ in Section \ref{S3_2}.)}
\ENDIF
\end{algorithmic}
\end{algorithm}

\section{Numerical Examples}\label{sec:results}

\subsection{Illustrative numerical example in d=2}\label{S4_1_1}

In this example we consider $d=2$ for a uniform weight on $\Gamma = [0,1]^2$ with an $r=2$ total degree polynomial space with index set $\Lambda_{\mathcal{T}_2}$. This index set has six indices, corresponding to six constraints in \eqref{PRO2_0}. Using $n=3$ nodes there are $(d+1)n = 9$ decision variables. Note that exact formulas for the optimal quadrature rule is known in this case \cite{stroud_numerical_1960}. The augmented Jacobian $\widetilde{\bs{J}}$ in \eqref{S3_3_6} is a $15 \times 9$ matrix. We initialize three nodes with a Latin hypercube design on $[0,1]^2$ and use uniform weights. The singular values of the Jacobian matrix are shown in
Figure~\ref{fig_1_-1} for the initial and final decision variables corresponding to three different choices of the regularization parameter $\lambda$. The results suggest that any positive value in $[0.01,5]$ can be used as a regularization parameter. We use a constant $\lambda$ throughout the iterations and fix the residual tolerance $\epsilon=10^{-8}$. The evolution of residual $\|\widetilde{\bs{R}}\|$, Newton decrement $\eta$, and penalty parameter $c_k$ is shown in Figure~\ref{fig_1_-1}. Smaller $\lambda$ values appear to yield faster convergence.
\begin{figure}[h]
\centering
\includegraphics[width=\textwidth]{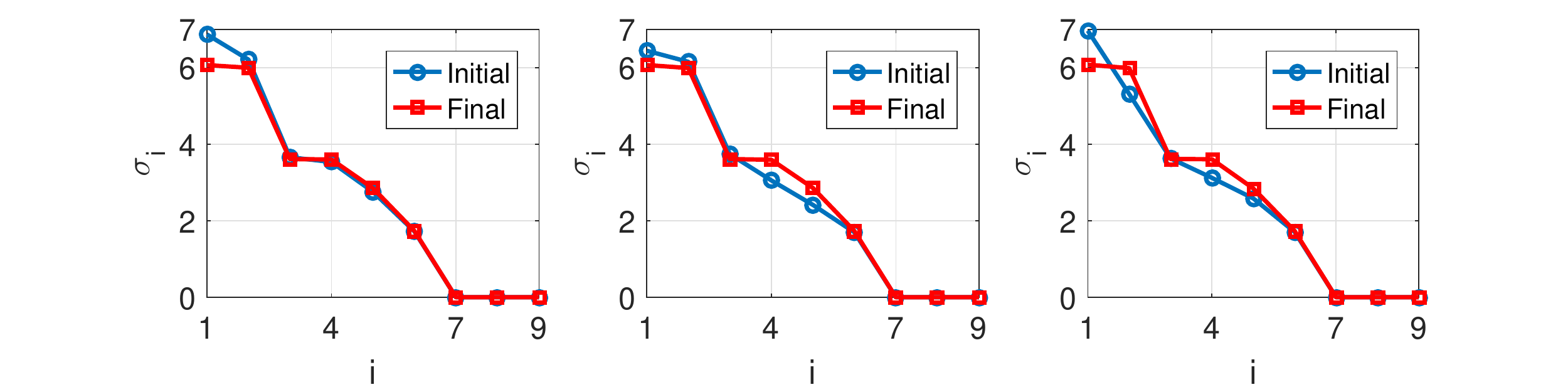}\\
\vspace{-0.12cm}
\includegraphics[width=\textwidth]{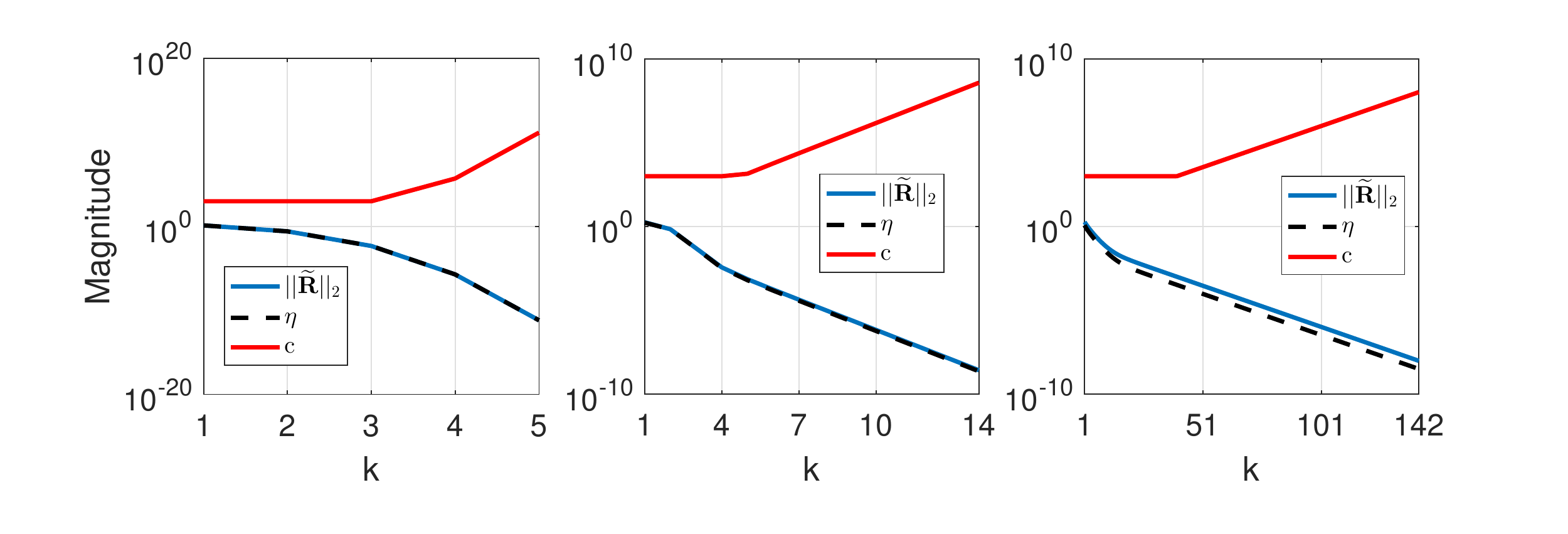}\\
\caption{\small{Singular values of Jacobian $\widetilde{\bm J}$ cf. Equation~\eqref{S3_3_6} (top) and residual norm $\|\widetilde{\bm R}\|_2$, Newton decrement $\eta$ and penalty parameter $c$ with respect to iterations $k$ (bottom)  for regularization parameter $\lambda=0.01$ (left),
$\lambda=1$ (middle) and $\lambda=5$ (right).}}\label{fig_1_-1}
\end{figure}

To visualize the optimal points for this quadrature, we randomize the initial node positions and compute designed quadrature for $100$ initializations. Plots of the ensemble of converged quadrature rules in two and three dimensions are shown in Figure~\ref{fig_1_0}.~\annote{A set of 3-points (initial and final design) in each experiment forms a triangle i.e. vertices of each triangle are the quadrature points where each triangle is visualized for distinguishment.} The cumulative time for $100$ designs took $\sim 6 ~sec$ with MATLAB on a single core personal desktop, and each design takes $\sim 15$ iterations with $\lambda=1$.

\begin{figure}[h]
\centering
\includegraphics[width=4.5in]{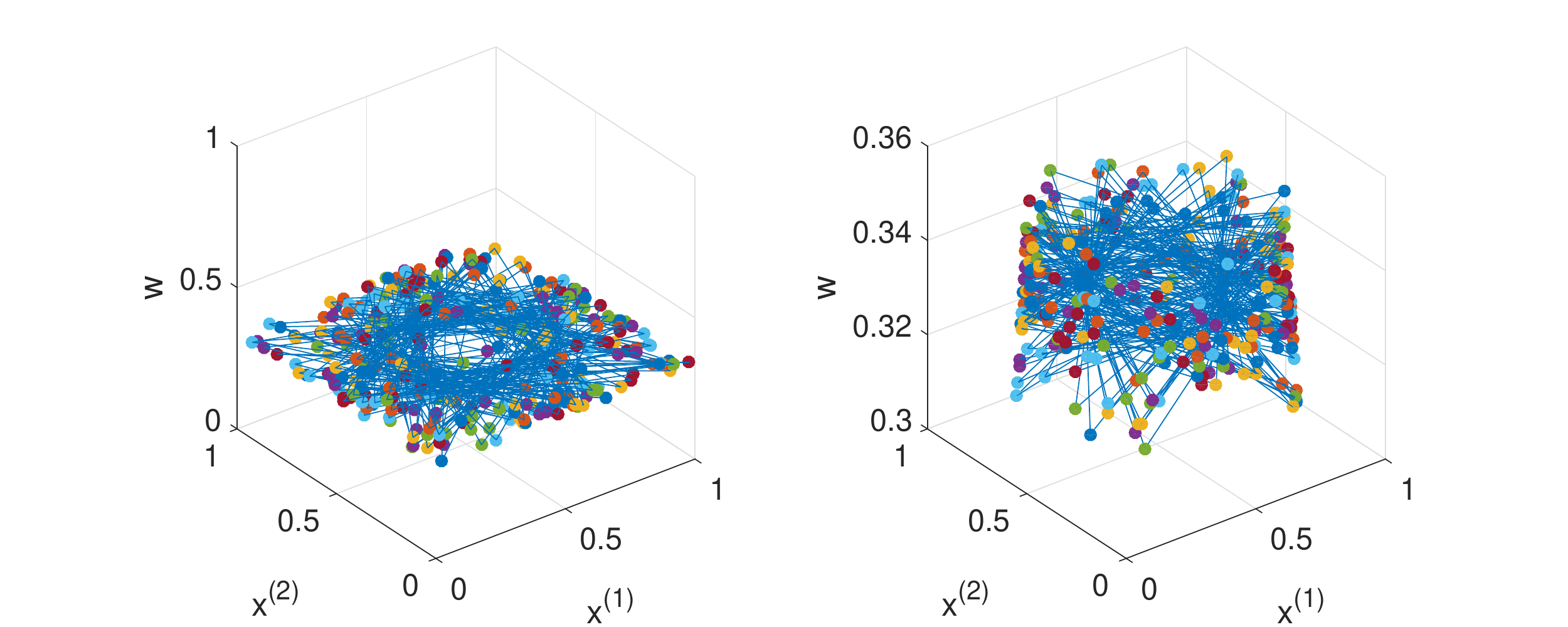}\\
\vspace{0.00cm}
\includegraphics[width=4.5in]{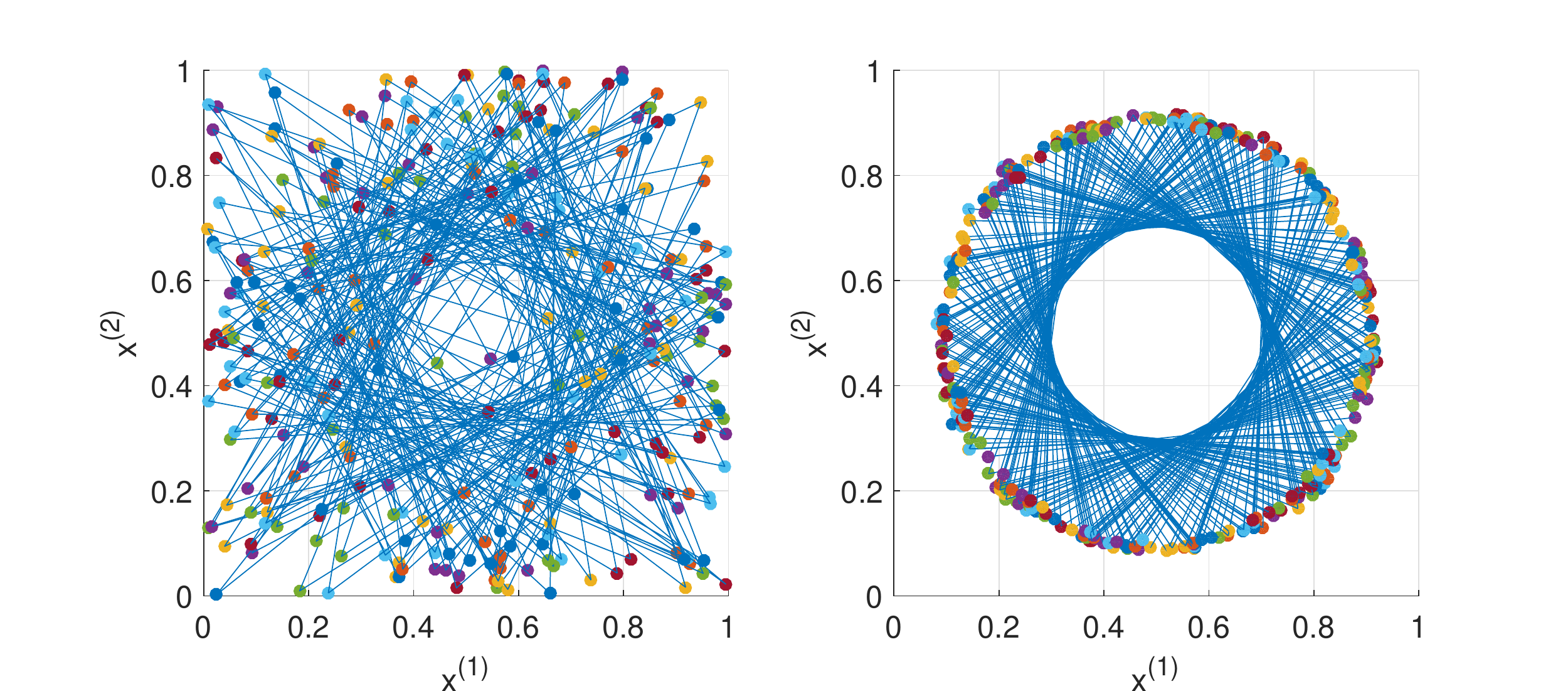}
\caption{\small{\annote{Ensemble of 3-point quadrature rules on $\Gamma = [0,1]^2$ found via designed quadrature ($d = r = 2$). Each three-point nodal configuration has nodes connected with blue lines, forming a triangle. Left: initial guesses provided to the algorithm. Right: converged designed quadrature rules. Bottom: Nodal configurations on $\Gamma$. Top: Weight values plotted as $z$-coordinates.}}}\label{fig_1_0}
\end{figure}

\subsection{Comparison with sparse grid quadrature}\label{sec:results-comparison}
In this example we consider the number of nodes required to achieve exact polynomial accuracy on total degree spaces $\Lambda_{\mathcal{T}_r}$ of various orders and dimensions. Our goal is to compare designed quadrature against sparse grids. The number of nodes required for exact integration on a sparse grid is from \cite{Qsparse}. Our tests fix dimension $d=3$ and sweep values of the order $r$, and fix $r=5$ and sweep values of the dimension $d$.
We present the nodal counts in Table~\ref{tabNE2_00} and in Figure~\ref{fig_2_0}. Table \ref{tabNE2_00} shows that designed quadrature consistently results in fewer nodes than sparse grids for moderate values of $r$ and $d$. We again emphasize that the weights for designed quadrature are all positive, unlike sparse grid quadrature.

Figure \ref{fig_2_0} compares various node counts: The number of nodes in the product rule is simply $n^d$ where $n$ is the number of univariate Gauss quadrature nodes and the ``lower bound" is the value $\mathcal{L}(\Lambda)$ determined from Theorem~\eqref{thm:half-set}. Using Theorem 2.1 in \cite{Jakeman17quadrature}, we can explicitly compute this as

\begin{align}\label{eq:lower-bound}
  \mathcal{L}\left(\Lambda_{\mathcal{T}_r}\right) = \left| \Lambda_{\mathcal{T}_{\lfloor r/2 \rfloor}}\right| = \left(\begin{array}{c} d+\left\lfloor r/2 \right\rfloor \\ d \end{array}\right)
\end{align}

Independently, we computed designed quadratures for $r=2$ and $r=3$ to confirm that the number of nodes for different dimensions $d$ coincides with $d+1$ and $2d$, respectively, as determined in \cite{Stroud71} (not shown). Also, for $r=5$ and $d=3,5$ we find the same number of nodes as those given by \cite{Stroud67} with positive weights.
\begin{table}[!h]
  \caption{Number of nodes $n$ sparse grids and designed quadratures on total degree spaces $\Lambda_{\mathcal{T}_r}$ on $\Gamma = [0,1]^d$. Top: fixed $d=3$ for various $r$. Bottom: Fixed $r=5$ for various $d$. \annote{The results in the top-half of this table can be compared with Table 4 in \cite{Xiao10}. Our quadrature rules have smaller or equal size compared with the results in \cite{Xiao10}, with the exception of $r=8$ where we report a 43-point rule instead of a 42-point rule in \cite{Xiao10}.}}\label{tab:comparison}
\centering
\resizebox{\textwidth}{!}{
\begin{tabular}{c c c c c c c c c c c c }
\hline
\hline
$d=3,r $  &1 &2& 3& 4 &5& 6& 7& 8& 9& 10& 11 \\
\hline
\textrm{Sparse Grid Quadrature (nested)}  & 1 & - & 7 &-& 19 &-&39&-&87&-&135\\
\textrm{Designed Quadrature}  & 1 & 4 &6 &10 &13 &22 &26 &43 &51 &74 &84\\
\hline\hline
$r=5,d$  &1 &2& 3& 4 &5& 6& 7& 8& 9& 10 \\
\hline
\textrm{Sparse Grid Quadrature (nested)}   &3& 9& 19& 33& 51& 73& 99& 129& 163& 201 \\
\textrm{Designed Quadrature}  &3 &7& 13& 21& 32& 44& 63& 88& 114& 148 \\
\hline
\end{tabular}
}
\label{tabNE2_00}
\end{table}

\annote{In Table~\ref{tabNE_peroformance} we show the performance of the scheme with respect to the number of nodes and iterations, CPU time (measured with tic-toc on MATLAB) and the achieved residual norm. To that end we consider $d=4$ for different orders $r$ and we set the tolerance to $\epsilon=10^{-12}$ in this example. It should be noted that these quantitative metrics can vary depending on the random initialization and regularization parameters throughout the algorithm however provide a useful holistic measure for the method's performance.
\begin{table}[!h]
\caption{Performance of the scheme with respect to number of nodes and iterations, CPU time and residual norm for $d=4$ and various orders $r$ with total order index set.}\label{tab:iterations}
\centering
\resizebox{\textwidth}{!}{
\begin{tabular}{c c c c c c c c c c c c }
\hline
\hline
$d=4,r $  &1 &2& 3& 4 &5& 6& 7& 8& 9& 10 \\
\hline
\textrm{Half-set size} $\left| \Lambda_{\mathcal{T}_{\lfloor r/2 \rfloor}}\right| $ & 1	& 5&	5&	15&	15&	35&	35&	70&	70&	126 \\
\textrm{Number of nodes}  &       1 &  5 & 8 & 16 &       21 &            43 &    55 &      103 &   138 &207\\
\textrm{Number of iterations}  & 10 & 9 &11 &56 &        179 &            146  & 298 &     153 &     461 &197\\
\textrm{CPU time (sec)}  &      0.09 &0.21 &0.25 &0.91 &   4.68&      8.45 & 30.66  &  33.06 &   183.67 &160.44\\
\textrm{Residual Norm } $||\tilde{\bm R}||_2$&       1e-14 &2e-14 & 4e-13 & 8e-13 &4e-13 &9e-13 & 9e-13&  3e-13 &    9e-13 &9e-13
\end{tabular}
}
\label{tabNE_peroformance}
\end{table}

}

\begin{figure}[h]
\centering
\includegraphics[width=4.7in]{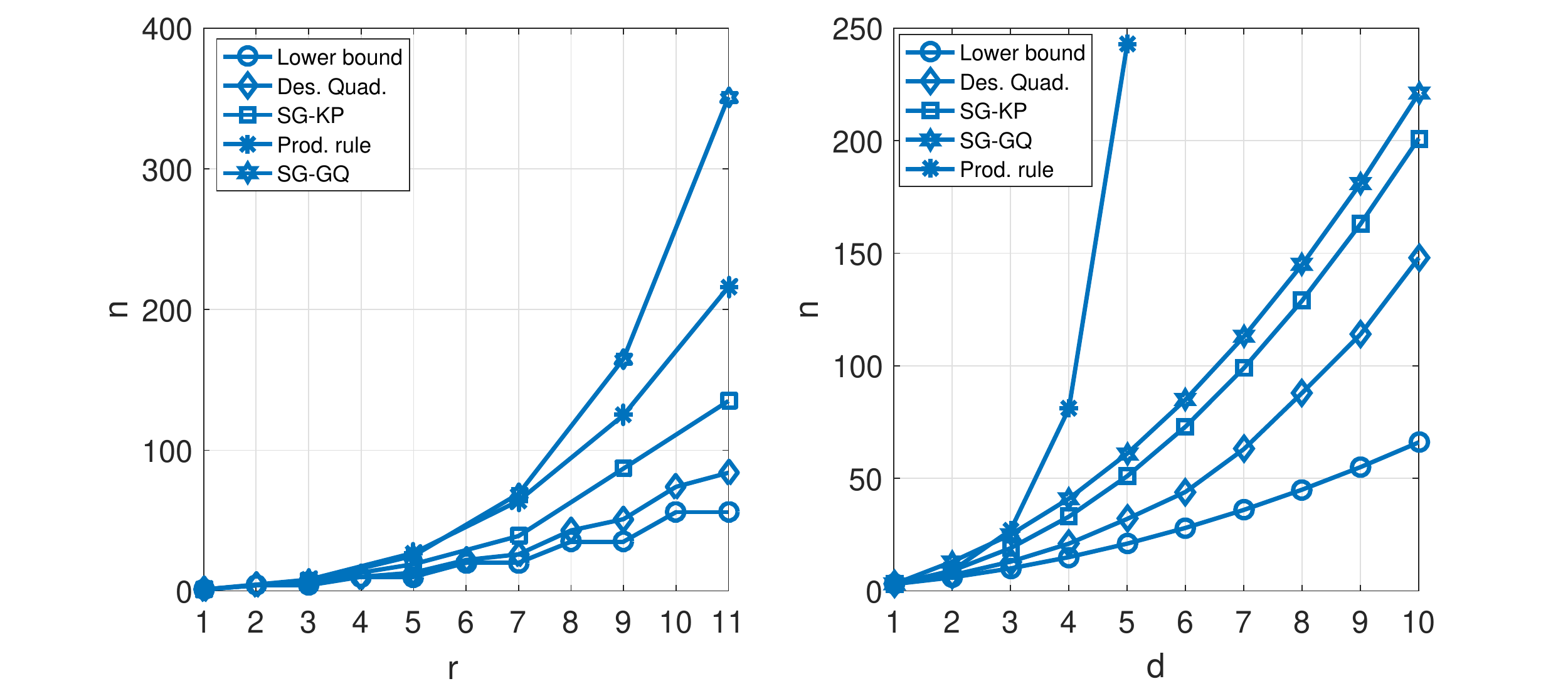}
\caption{\small{Number of nodes for fixed $d=3$ (left), and fixed $r=5$ (right) for total order index set $\Lambda_{\mathcal{T}_r}$. The lower bound is given in \eqref{eq:lower-bound}, ``SG-KP" and ``SG-GQ" are sparse grid constructions using nested Kronrod-Patterson and non-nested Gauss quadrature rules, respectively.}}\label{fig_2_0}
\end{figure}

To illustrate how the regularization parameter $\lambda$ is chosen, we show the singular values and regularization parameter choice for the case $r=5, d=7$. Figure~\ref{fig_2_01} shows the regularization parameter selection for an iteration in the middle of the procedure. The regularized parameter is selected as $\lambda=10$ by investigating the spectrum of singular values and its L-curve.

In practice, one could fix $\lambda$ as a function of $\|\bs{R}\|_2$ (or $\|\widetilde{\bs{R}}\|_2$). In Figure \ref{fig_2_0}, we have $\lambda = 10$ with $\|\bs{R}\|_2 = 40$. Then, for example, one could take $\lambda=50$ for $200 \leq ||\bs{R}||_2 \leq 500$ and $\lambda=10$  for $20 \leq ||\bs{R}||_2 \leq 200$. Such an \textit{a priori} tabulation could be fixed for a variety of $(d,r)$ values.

\begin{figure}[h]
\centering
\includegraphics[width=4.5in]{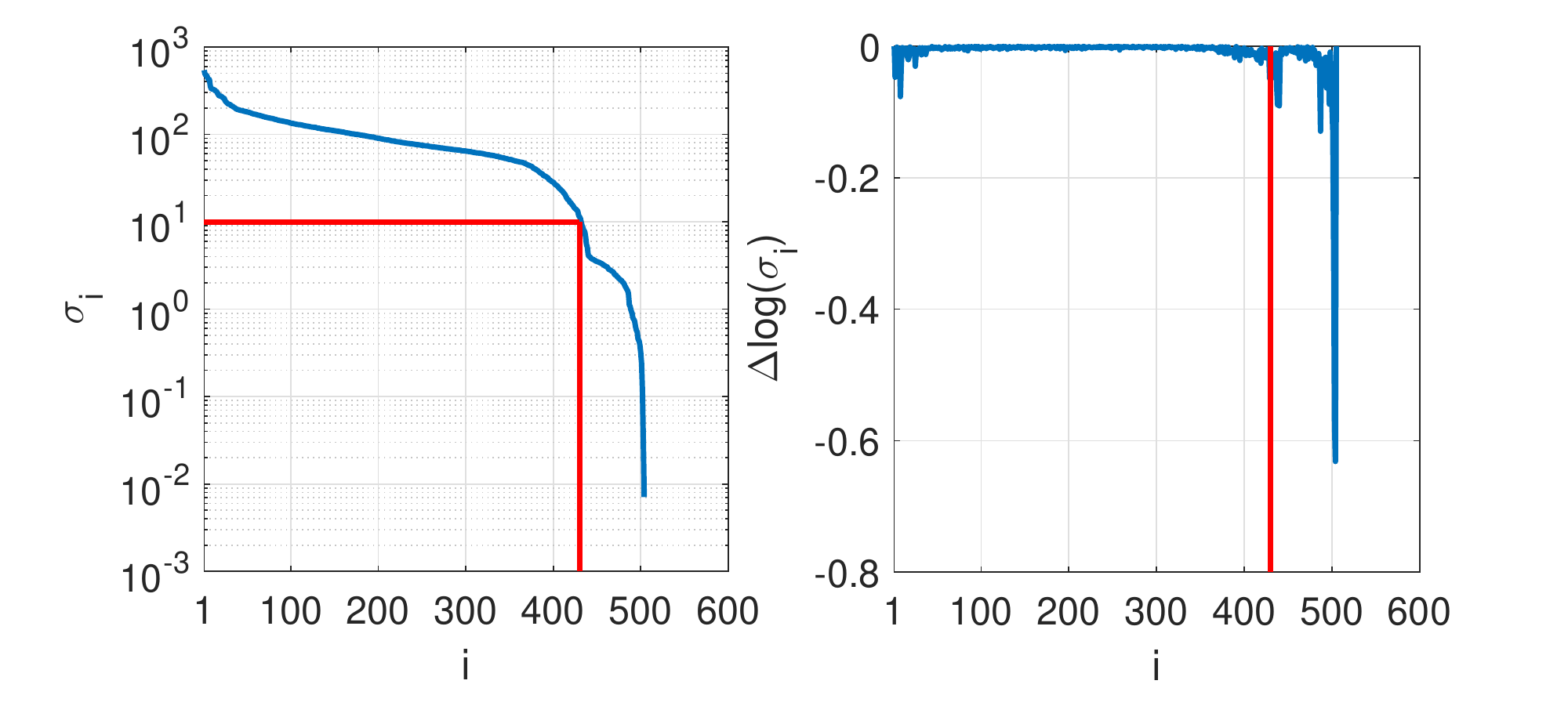}\\
\includegraphics[width=4.5in]{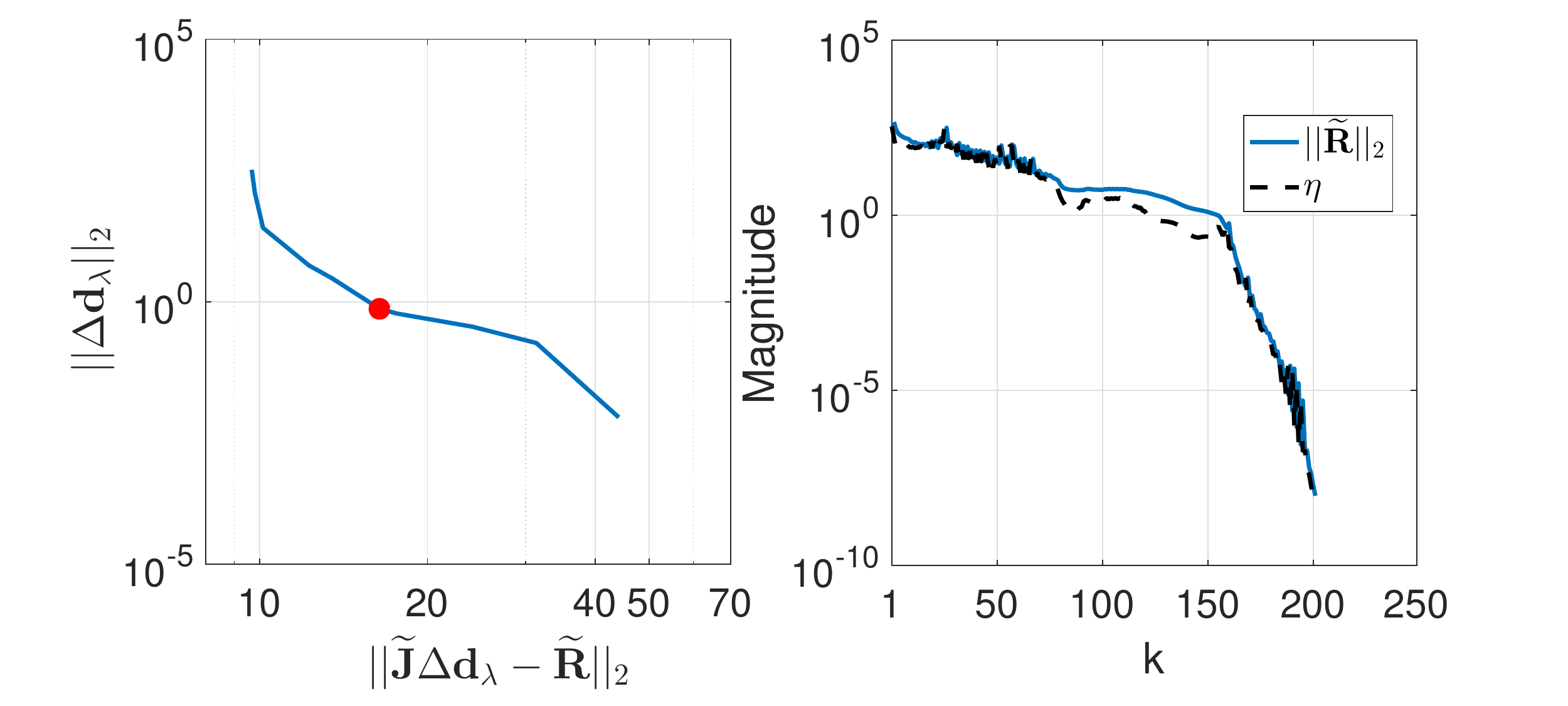}
\caption{\small{Singular values of $\widetilde{\bm J}$ with the chosen regularization parameter (top left), finite difference on $\log$ of singular values and the chosen singular value index (top right), $L$-curve for the given $\widetilde{\bm J}$ and $\widetilde{\bm R}$: the point on the $L$-curve corresponding
to the selected regularization parameter $\lambda$ is indicated with the red circle (bottom left), convergence of the scheme for $d=5,r=7$ (bottom right).}}\label{fig_2_01}
\end{figure}

\subsection{Interpolation with designed quadrature}
Designed quadrature rules can be used to construct polynomial interpolants. Suppose we have a designed quadrature rule $(\bs{X}, \bs{w})$ of size $n$ that matches moments for indices on $\Lambda$ (up to the tolerance $\epsilon$), and assume $n = |\Lambda|$\footnote{Designed quadrature rules achieve $n < |\Lambda|$, but in this section we will enforce $n = |\Lambda|$ for the purposes of forming an interpolant.}. For continuous function $f$, let $\mathcal{I}(f)$ denote the unique interpolant of $f$ from $\Pi_{\Lambda}$ at the locations $\bs{X}$. Lebesgue's lemma states
\begin{align*}
  ||f-\mathcal{I}(f)||_{\infty} &\leq (L+1) \inf_{p \in \Pi_{\Lambda}} ||f-p||_{\infty}, & L &= \sup_{\|h\|_\infty = 1} \left\| \mathcal{I}(h) \right\|_\infty,
\end{align*}
where $\|\cdot\|_\infty$ is the maximum norm on $\Gamma$, and the supremum is taken over all functions $h$ continuous on $\Gamma$. The constant $L$ is the Lebesgue constant; small values indicate that interpolants are comparable to the best approximation measured in the maximum norm \annote{\cite{lubinsky_survey_2007}}. The Lebesgue constant can be computed explicitly: The interpolant $\mathcal{I}(f)$ can be expressed as
\begin{align*}
  \mathcal{I}(f)(\bs{x}) &= \sum_{j=1}^n \ell_j(\bs{x}) f(\bs{x}_j), & \ell_j(\bs{x}_k) &= \delta_{j,k},
\end{align*}
where the $\ell_j$ are the cardinal interpolation functions. \revv{The Lebesgue function $L_n$} and the Lebesgue constant $L$ are, respectively,
\revv{
\begin{align*}
  L_n(\bm x) &= \displaystyle \sum_{j=1}^n |\ell_j(\bm x)|, & L = \|L_n\|_\infty
\end{align*}
}

Finding a set of points with minimal Lebesgue constant is not trivial. In $d=2$ dimensions the Padua points are essentially the only explicitly constructible set of nodes with provably minimal growth of Lebesgue constant on total degree spaces \cite{Bos06}. To compare designed quadrature with Padua points, we consider degree-$5$ Padua points, yielding $\dim \Pi_{\Lambda_{\mathcal{T}_5}} = 21$. These points along with associated quadrature weights integrate polynomials in $\Pi_{\Lambda_{\mathcal{T}_9}}$ exactly with respect to the product Chebyshev weight $\omega$ \cite{Bos06}.

With designed quadrature we are able to find $17<21$ nodes and weights that integrate polynomials in $\Pi_{\Lambda_{\mathcal{T}_9}}$ exactly. However, for the purposes of interpolation in this section, we enforce $n = 21$ nodes in the designed quadrature framework. To initialize the design we start from nodes that are close to Padua nodes. The Lebesgue function $L_n(\bs{x})$ for both cases are shown in Figure~\ref{fig_5_1}. The Lebesgue constant for Padua points and designed quadrature are $L=4.9478$ and $L=5.1553$, respectively. The similar small values of $L$ suggest that the designed quadrature points and the Padua points are of comparable quality in terms of constructing interpolants. However, we reiterate that for quadrature we can use fewer nodes (17) than the Padua points (21).

\begin{figure}[h]
\centering
\includegraphics[width=\textwidth]{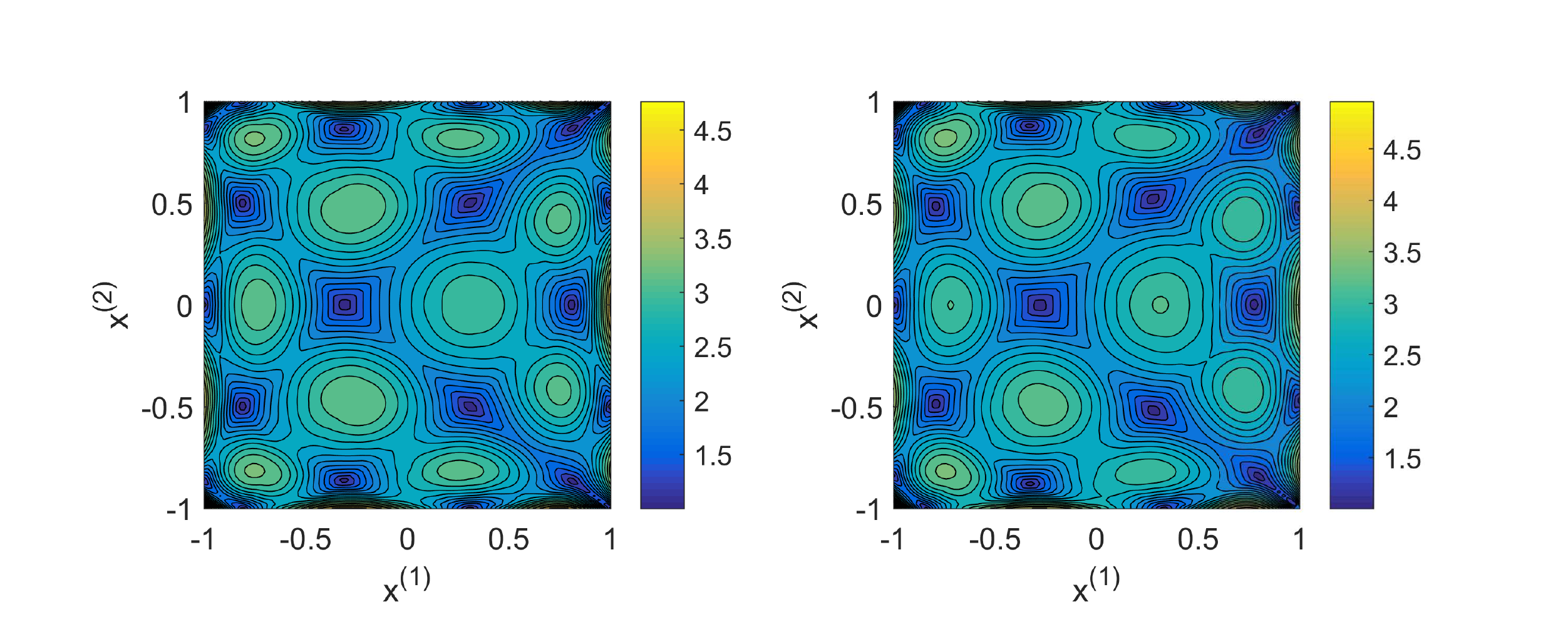}
\vspace{-0.3cm}
\caption{\small{Contour plots of Lebesgue function $L_n$ for Padua points (left) and designed quadrature (right)}}\label{fig_5_1}
\end{figure}

\subsection{Designed quadrature: U}\label{sec:results-U}
The formulation of designed quadrature allows $\Gamma$ and $\omega$ to be of relatively general form, but we can construct quadrature rules in even more exotic situations. Let $\Gamma = [-1,1]^2$ with $\omega$ the uniform weight. Instead of enforcing $\bs{x}_j \in \Gamma$ in \eqref{S3_3_3}, we enforce $\bs{x}_j \in \widetilde{\Gamma}$, where $\widetilde{\Gamma} \subset \Gamma$ is a ``U" shape, mimicking the logo of the University of Utah; see Figure \ref{fig_5_2}, left.

The penalty function for this problem has the same quadratic form as those discussed in Section~\ref{S3_3} and separate penalties are considered for violations in both $x^{(1)}$ and $x^{(2)}$ directions. For example, we can model the infeasible rectangular region $\mathcal{S}_1$ between the two ascenders of the U with non-zero penalty in $x^{(1)}$ direction and zero penalty in $x^{(2)}$ direction as

\begin{equation}\label{S3_3_4}
\begin{array}{l }
  \begin{cases}
\mathcal{S}_1 :  (0 \leq |x_1^{(1)}| \leq 0.4) \cap (-0.35 \leq x^{(2)} \leq 0.95) \\
P_1 = (x^{(1)}-0.4)^2, P_2=0
\end{cases}
\end{array}
\end{equation}
A similar method can be used to penalize the semicircular region below the rectangle where violations in both directions are penalized. The total penalty for infeasible regions then involves both $P_1$ and $P_2$ e.g.  $P=10\sqrt{P_1+P_2}$ which is shown in Figure \ref{fig_5_2}, right.  We compute a designed quadrature rule for total degree $r=2$, achieving residual tolerance of $\epsilon=0.0098$ with $n=150$. We need a relatively large number of nodes, and achieve only a relatively large tolerance (compared to $\epsilon=10^{-8}$ for previous examples). This is due to the difficulty of this problem: we want nodes to lie in $\widetilde{\Gamma}$ but want to achieve integration over $\Gamma$. We expect that convergence for larger $r$ will require many iterations and may not be able to achieve arbitrarily small tolerances.
\begin{figure}
\centering
\begin{minipage}{0.45\textwidth}
\includegraphics[width=\textwidth]{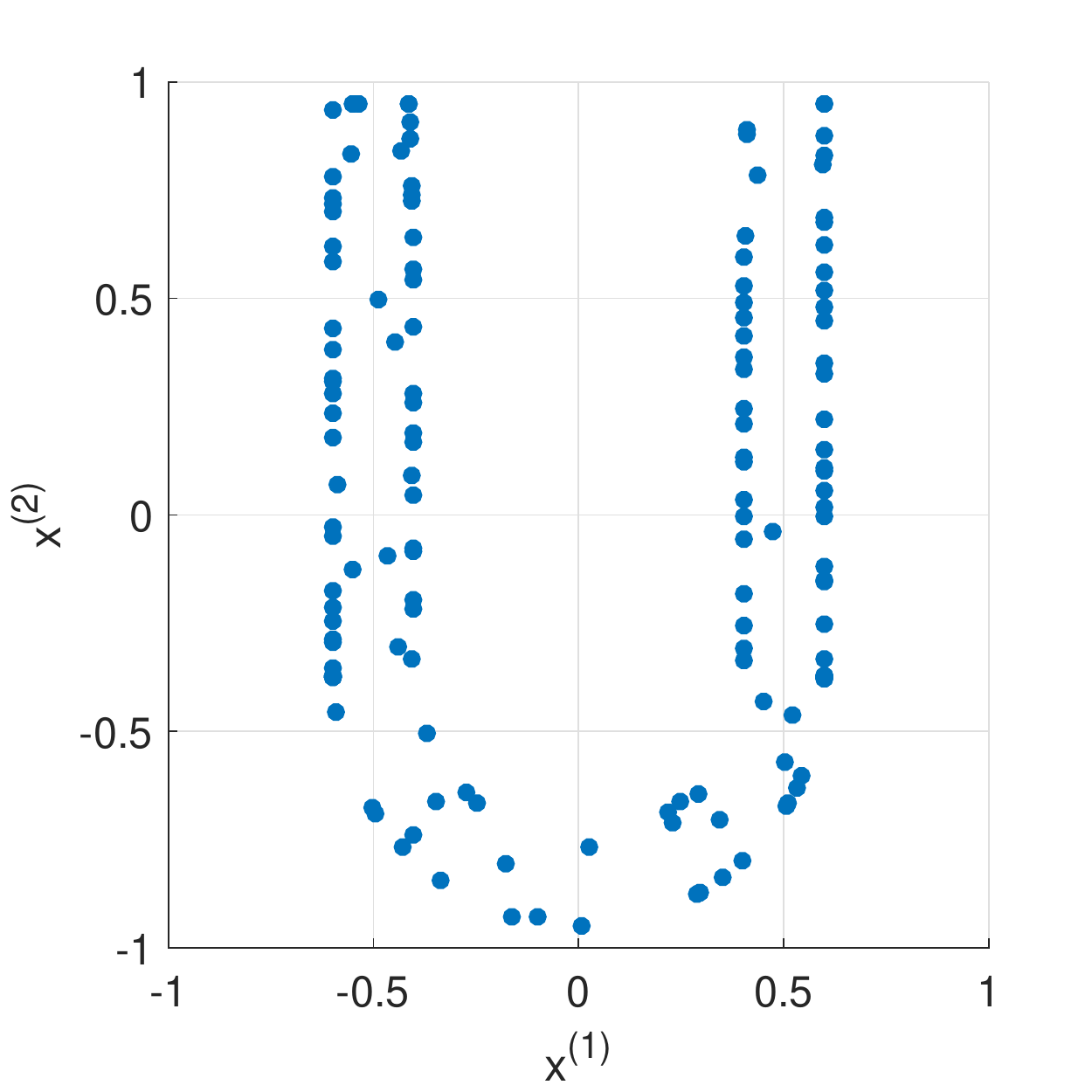}
\end{minipage}
\begin{minipage}{.45\textwidth}
\includegraphics[width=\textwidth]{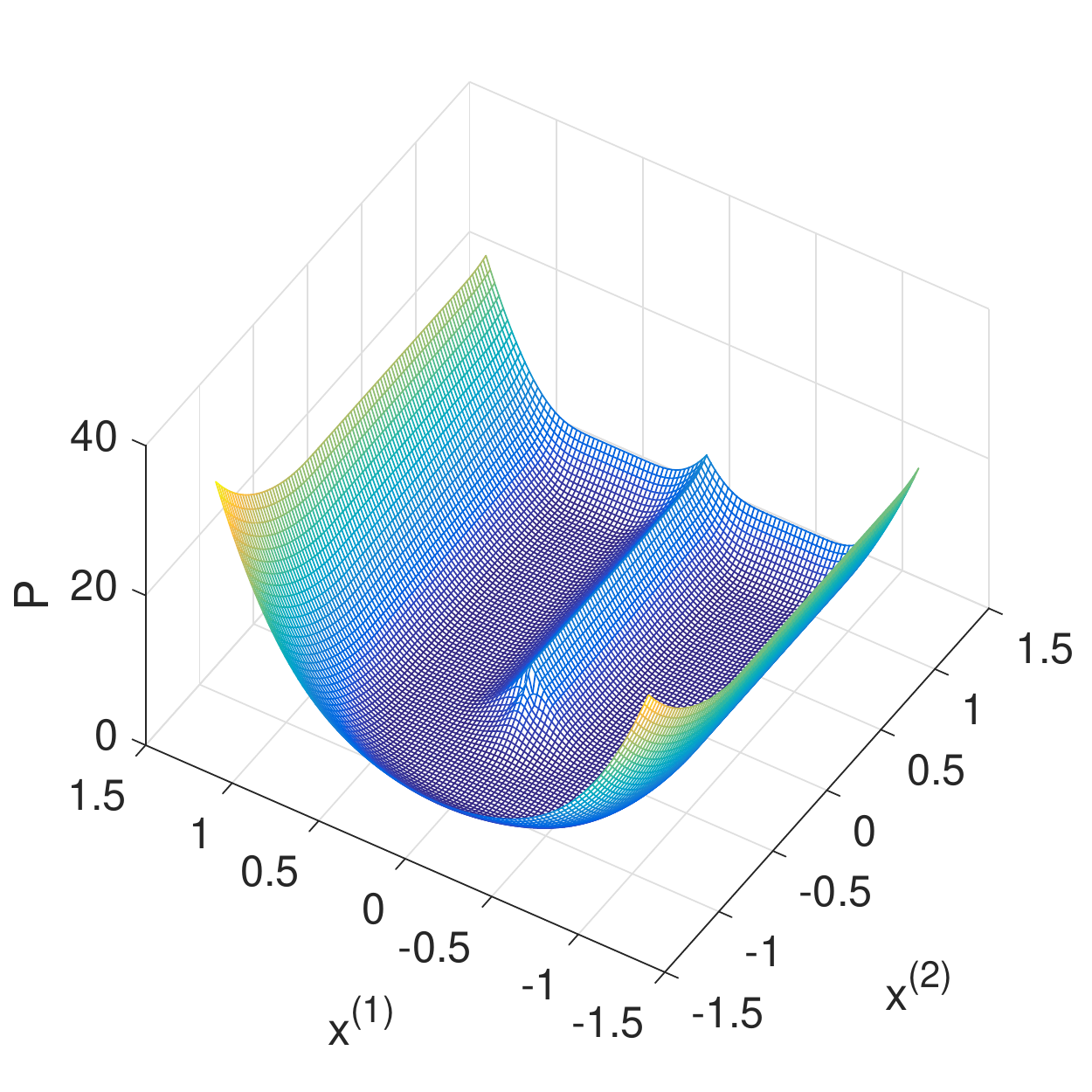}
\end{minipage}
\vspace{-0.1cm}
\caption{\small Designed Quadrature for uniform weight and $d=r=2$ with ``U'' shape indicating the University of Utah (left), the penalty function used in the scheme (right).}\label{fig_5_2}
\end{figure}

\subsection{Integration in high dimensions}\label{sec:results-highd}
To demonstrate the capability of designed quadrature for integration in high dimensions, we consider $d=100$ with hyperbolic cross index set $\Lambda_{\mathcal{H}_r}$. Figure~\ref{fig_2_1} (top) shows different slices of the $d=100$ nodal configuration generated by designed quadrature for the uniform weight on $\Gamma = [-1,1]^{100}$ and $r=4$, for which we have $n=106$ and $|\Lambda_{\mathcal{H}_4}|=5351$.

Figure \ref{fig_2_1} (bottom) shows the behavior of designed quadrature weights with respect to the Euclidean norm of the nodes (distance to the origin) for $\omega$ the Gaussian weight on $\Gamma = \R^{100}$ for total order $r=2$ and hyperbolic cross orders $r=3,4$. As expected the weights decay as the node norms increase. We find
$n=101$ for all these quadratures, again confirming the optimal $n=d+1$ size for total order $r=2$ \cite{xiu_numerical_2008}. It is also interesting to note that the minimum Euclidean norm of nodes for these cases are somewhat equal viz $||x||=8.89,~8.93,~8.85$ respectively.

Computing designed quadratures in high dimensions reveals computational challenges that are not present in small-to-moderate dimensions: since the nonlinear system is quite large, we do not perform the SVD of Jacobian in each iteration. Instead, we regularize the pseudo-inverse matrix directly and compute the Newton step as $\Delta \bm d = (\bm J^T \bm J+\lambda \bm I)^{-1} \bm J^T \bm R$. The parameter $\lambda$ can be selected based on the residual norm value, as explained in the previous section.
The designed quadrature algorithm for these cases in $d=100$ took $\sim 100$ iterations and less than $30$ minutes on a personal desktop in MATLAB.
\begin{figure}[h]
\centering
\includegraphics[width=\textwidth]{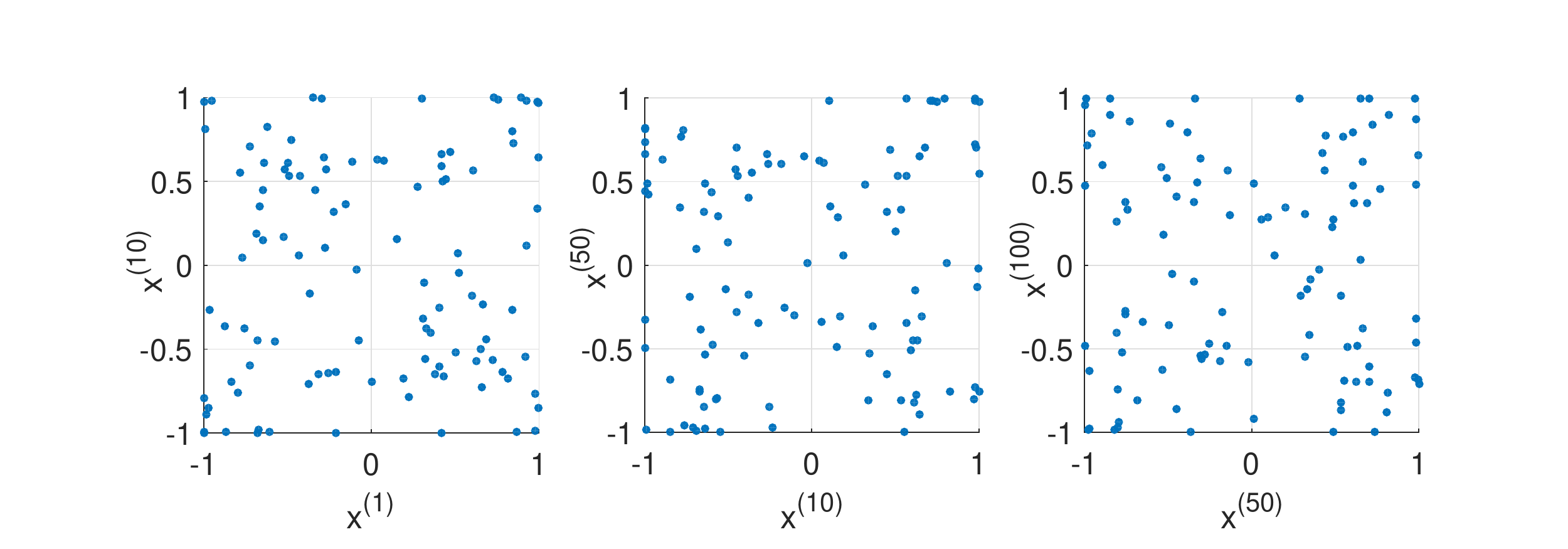}\\
\vspace{-0.2cm}
\includegraphics[width=\textwidth]{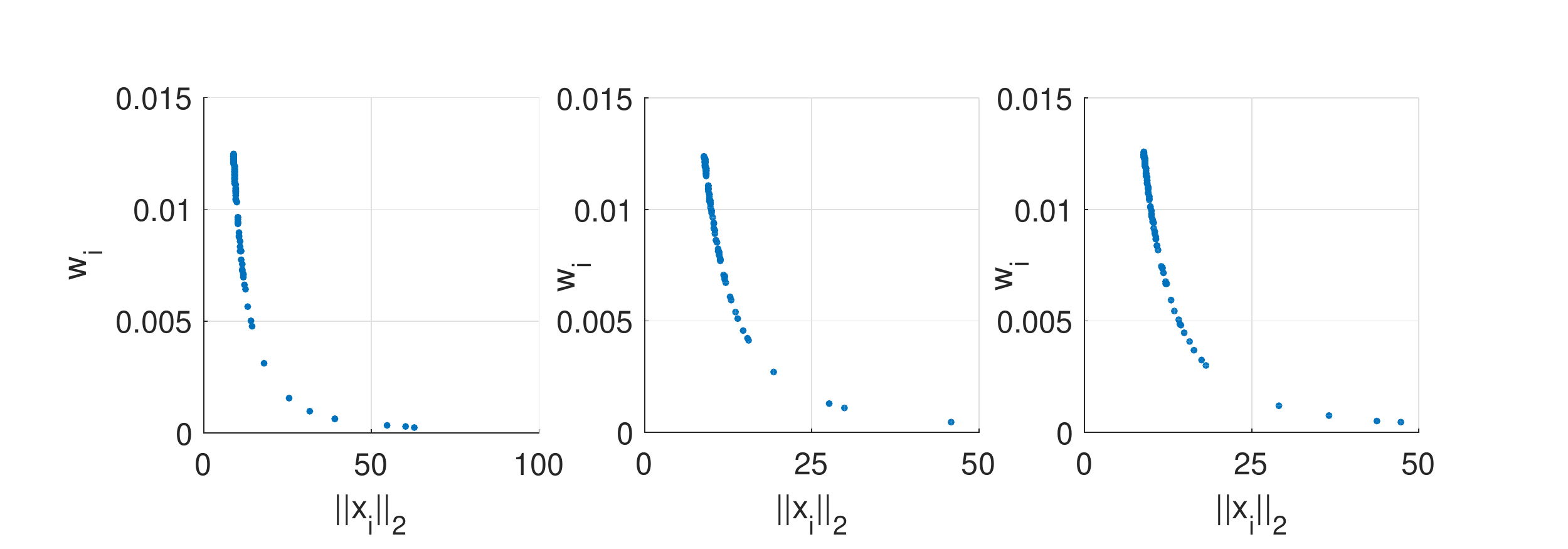}
\caption{\small{Different two dimensional slices of $d=100$ designed quadrature for uniform weight and hyperbolic cross order $r=4$ (top), formation of weights with respect to Euclidean norm of nodes for Gaussian weight in $d=100$ (bottom): total order
$r=2$ (left) hyperbolic cross order $r=3$ (middle) and $r=4$ (right).}}\label{fig_2_1}
\end{figure}

\subsection{High-dimensional integration: linear elasticity problem}
To investigate the performance of the high dimensional quadrature points we compute the mean and variance of compliance indicative of the elastic energy for a solid cantilever beam with uncertain material properties.

\annote{ The compliance for the spatial domain $\Omega$ reads
\begin{align*}
C = \int_{\Omega} fu d\Omega
\end{align*}
where $u$ is the displacement and $f$ is the surface load on the structure. To find $u$, the equation of motion in linear elasticity $\nabla . \bm \sigma + f = 0$ where $\bm \sigma$ is the stress tensor and $\nabla$ is the divergence operator is solved via Finite Element Method.  The global displacement is characterized with $n_e$ finite elements
\begin{align*}
  u = \sum_{i=1}^{n_e} u_i \Psi_i
\end{align*}
where $\Psi_i$ are finite element shape functions and $u_i$ are nodal displacements. The nodal displacements $\bm U=\{u_i\}_{i=1}^{n_e}$ are solution of a linear system $\bm K \bm U = \bm F$ (stems from the equation of motion) where
\begin{align*}
\begin{array}{l}
\bm K=\displaystyle \int_{\Omega} \displaystyle\frac{\partial \Psi^T}{\partial \bm {x}} \mathbb{C} \displaystyle\frac{\partial \Psi}{\partial \bm {x}} d \Omega, \\
\\
\bm F= \displaystyle \int_{\Omega} f \Psi d \Omega
\end{array}
\end{align*}
with $\mathbb{C}$ being an elasticity matrix. We consider the plane stress condition in this example, hence for our two dimensional problem

\begin{align*}
\mathbb{C} =\frac{E}{1-\nu^2}
\left[
\begin{array}{ccc}
  1 & \nu & 0 \\
  \nu & 1 & 0 \\
  0 & 0 & \displaystyle \frac{1-\nu}{2}
\end{array}
\right]
\end{align*}
where $E$ is the modulus of elasticity and $\nu$ is the Poisson's ratio.
}

The beam geometry is shown in Figure~\ref{fig_2_2}, and is modeled with $100$ standard square finite elements, where each element has lognormal modulus of elasticity as $E_i=10^{-9}+\exp({\xi^{(i)}})$. The random variables $\xi^{(i)}|_{i=1}^{100}$ are independent standard random normal variables and the Poisson's ratio is $\nu=0.3$. Our goal is to compute first- and second-order statistics of the compliance; these statistics are integrals with respect to the 100 variables $\xi^{(i)}$, and so we approximate these statistics via designed quadrature.

For comparison against designed quadrature we use quasi-Monte Carlo (QMC) samples of size \annote{$n=64,128,256,512$ and $n=1024$}, where we treat the latter as the exact solution. The QMC samples are generated on $[0,1]^{100}$, and are mapped to $\R^{100}$ using inverse transform sampling for a 100-dimensional standard normal random vector. We have chosen the number of QMC samples so that they \annote{almost} match the number of designed quadrature nodes computed from the index sets (i) total order with $r=2$ \annote{and $n=101$ nodes}, (ii) the index set $\Lambda_{\mathcal{H}_4} \cup \Lambda_2$ \annote{with $n=155$ nodes}, where $\Lambda_2$ contains pairwise interactions of maximum univariate order 2, and (iii) the index set $\Lambda_{\mathcal{H}_4} \cup \Lambda_3$  \annote{with $n=255$ nodes}, where $\Lambda_3$ contains pairwise interactions of maximum univariate order 3.

Figure~\ref{fig_2_2} compares errors in the computed mean and standard deviation of the compliance for QMC versus designed quadrature; we see that designed quadrature achieves significantly better errors. We believe QMC would be more effective if the problem involved many more variables, larger index sets, \annote{and/or non-smooth quantities of interest}, as in those cases the number of designed quadrature is prohibitively large and finding a suitable quadrature rule is computationally challenging.

\begin{figure}
\centering
\begin{minipage}{0.450\textwidth}
\includegraphics[width=\textwidth]{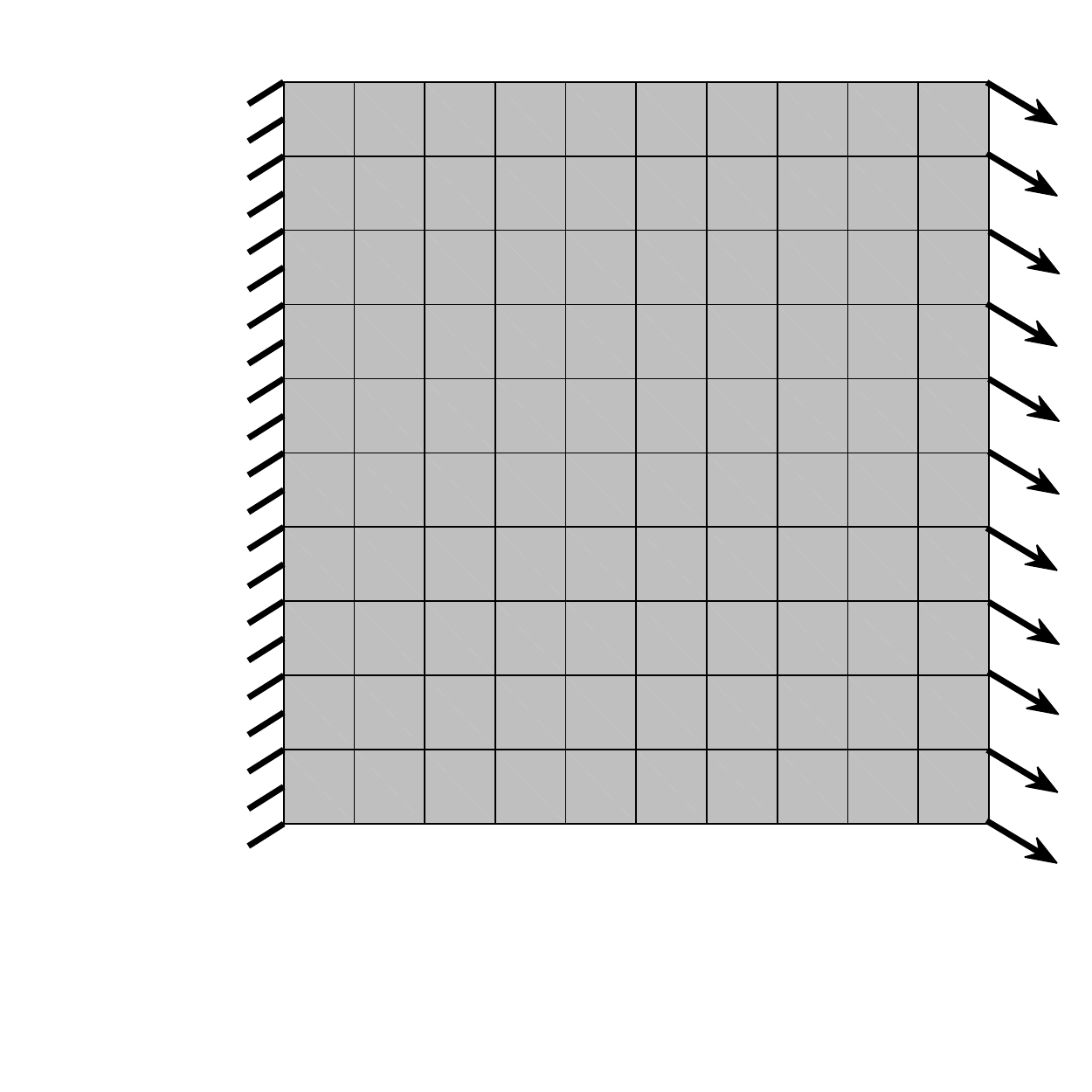}
\end{minipage}
\begin{minipage}{.43\textwidth}
\vspace{-0.75cm}
\includegraphics[width=\textwidth]{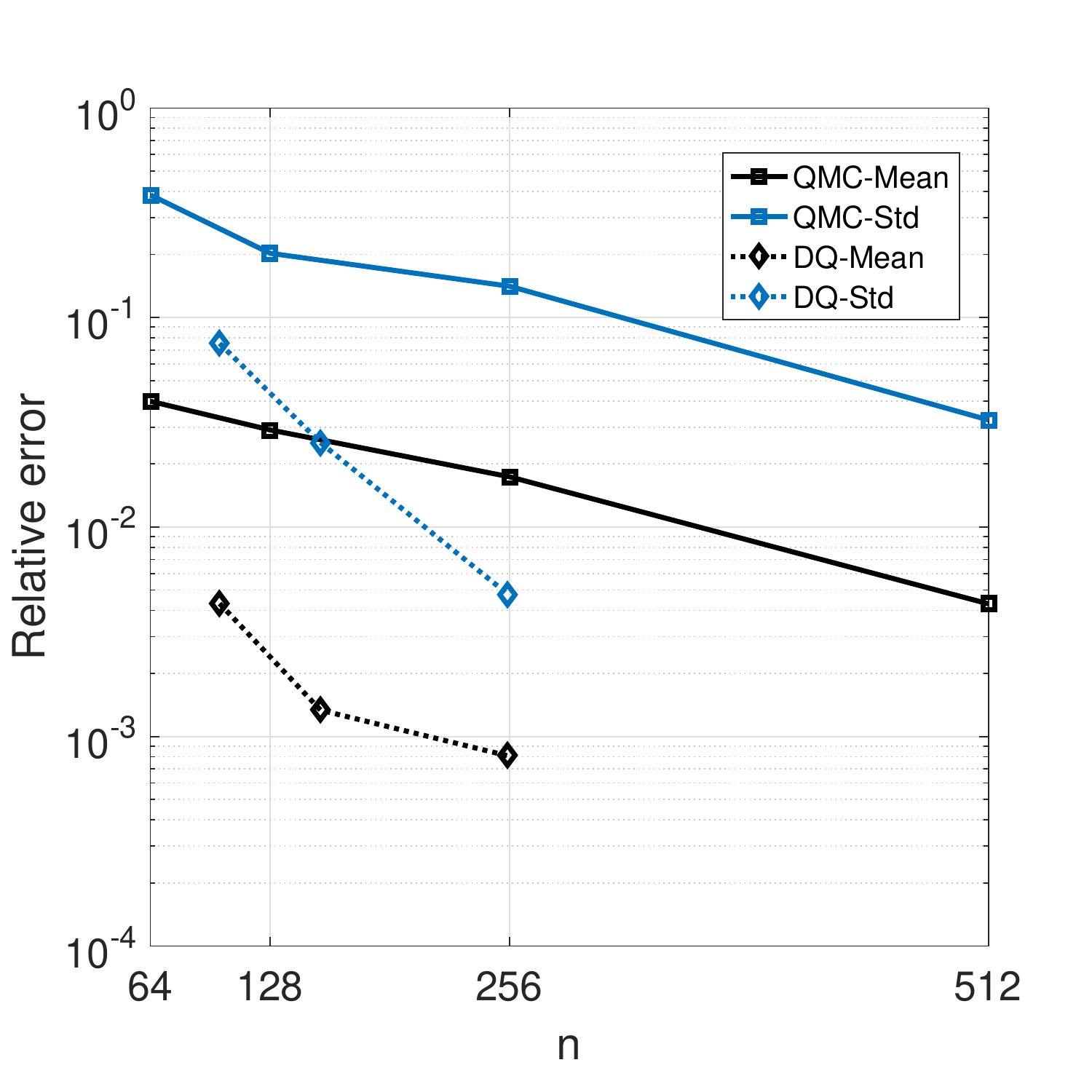}
\end{minipage}
\vspace{-0.7cm}
\caption{\small Finite element discretization for a linear elastic cantilever beam with random elastic modulus (left) convergence of compliance mean and variance with quasi-Monte Carlo (QMC) samples and designed quadrature (DQ) in $d=100$ variables
(right).}\label{fig_2_2}
\end{figure}

\subsection{Designed quadrature for topology optimization under uncertainty}
Our final example utilizes polynomial chaos (PC) methods \cite{ghanem02} to build surrogates for topology optimization under geometric uncertainty \cite{Keshavarzzadeh17}. Figure~\ref{fig_6_0_0} shows the flowchart for design optimization under uncertainty. To build PC surrogates at each design iteration, $n$ Finite Element Analysis (FEA) and sensitivity analyses are performed in order to quantify the uncertainty associated with random variables $\xi^{(i)}$ as in the last example. This is the most costly step in the design process, and hence small $n$ can result in significant computational savings.
\begin{figure}
\centering
\includegraphics[width=4.0in]{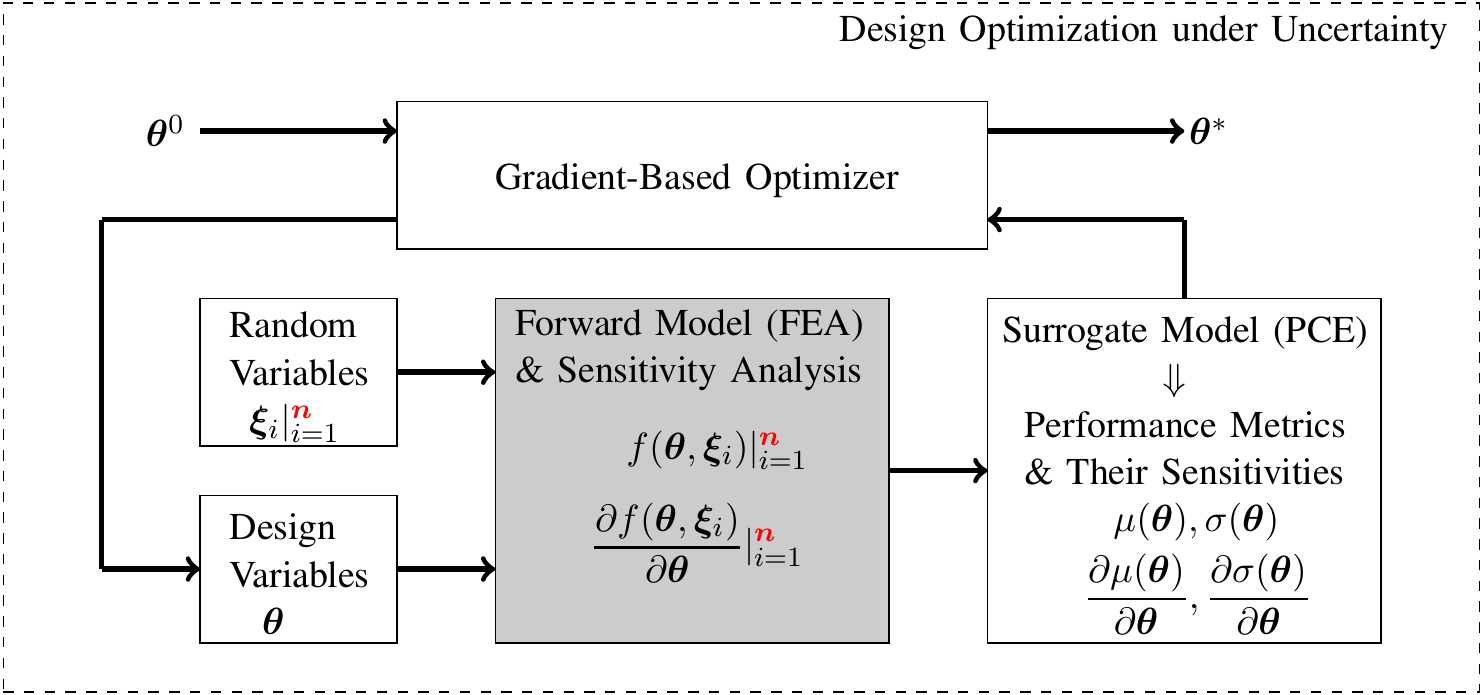}
\caption{\small Design optimization under uncertainty flowchart}\label{fig_6_0_0}
\end{figure}

The perturbation in the boundary of topology interfaces $Z$ are modeled via a Karhunen-Lo\`{e}ve random field with $d=4$ significant modes as \begin{equation}\label{S5_3}
Z(x,\bm \xi) = \displaystyle \sum_{i=1}^4 \sqrt{\lambda_{i}} \gamma^{(i)}(x) \xi^{(i)}.
\end{equation}
where $\xi^{(i)} \in U[-\sqrt{3},\sqrt{3}]$ are independent uniform random variables. Sparse grids built from nested rules were utilized in \cite{Keshavarzzadeh17} to develop a surrogate for total degree $r=3$ in $d=4$ dimensions. This requires a quadrature rule that can accurately integrate polynomials up to order $r=6$ (see Proposition \ref{prop:quadrature-stability}). A standard construction of sparse grid rules yields odd orders of polynomial accuracy and hence a quadrature rule for $r=7$ with $n=81$ nodes is used. We observe that $33$ out of these $81$ nodes have negative weights \cite{Qsparse}. On the other hand we use designed quadrature constrained to integrate polynomials up to degree $r=6$ and compute $n=43$ nodes, almost half ($\sim 53\%$) the number of sparse grid points (43/81=0.53) and all nodes have positive weights. \annote{These nodes and weights are listed in Table~\ref{node_weight_print}.}

We approximate the mean and variance for the final robust topology design of the Messerschmitt-B\"{o}lkow-Blohm (MBB) beam shown in Figure~\ref{fig_6_0} (left) with both quadrature sets, and use a sparse grid rule with $n=641$ points ($r=13$) as the ``true" solution.  The mean and standard deviation for the true solution are $\mu_{true}=120.1032, \sigma_{true}=0.7853$ respectively. The mean, standard deviation and relative errors in mean $e_{\mu}=|(\mu-\mu_{true})/\mu_{true}|$ and in standard deviation $e_{\sigma}=|(\sigma-\sigma_{true})/\sigma_{true}|$ are listed in Table~\ref{tabNE2_30}. We achieve higher accuracy with designed quadrature at nearly half the cost. Figure~\ref{fig_6_0} also visually compares the probability density function (PDF) of compliance for both cases, and no substantial difference is observed.

\begin{table}[!h]
\caption{Mean and standard deviation estimation for the robust topology design.}
\centering
\begin{tabular}{l c c c c c c }
\hline\hline
\textrm{Quadrature Rule}   & $\mu$ & $e_{\mu}$ & $\sigma$ &  $e_{\sigma}$ & Cost \\
\hline
\textrm{Sparse Grid}   & 120.1326 & 2.44e-04 & 0.7836 & 2.16e-03 & 81 Simulations \\
\textrm{Designed Quadrature}  & 120.1028 & 3.33e-06 & 0.7854 & 1.27e-04 & 43 Simulations \\
\hline
\end{tabular}
\label{tabNE2_30}
\end{table}

\begin{figure}
\centering
\begin{minipage}{0.40\textwidth}
\includegraphics[width=\textwidth]{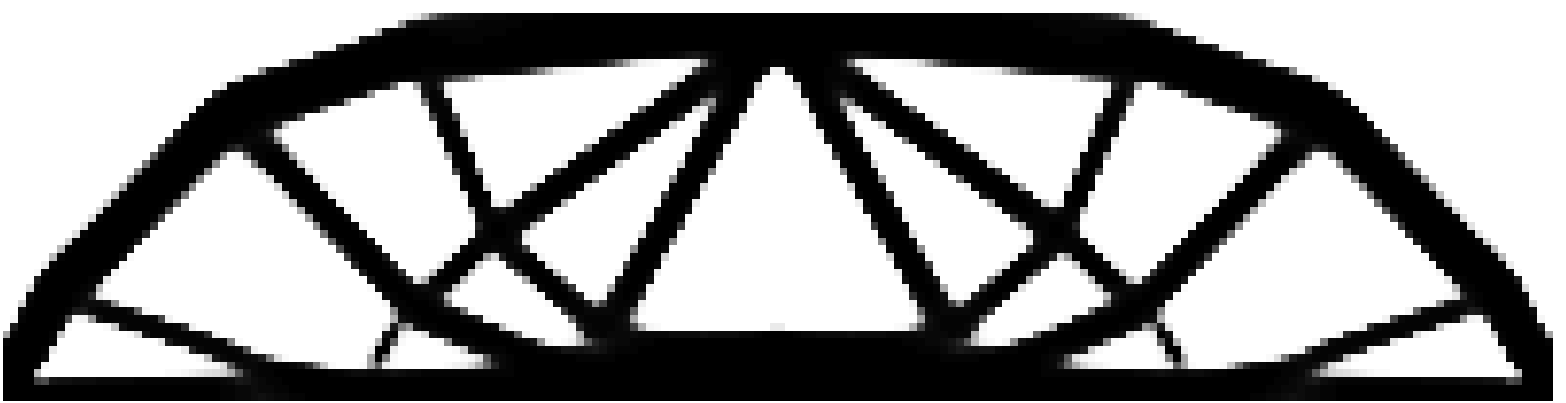}
\end{minipage}
\begin{minipage}{.40\textwidth}
\includegraphics[width=\textwidth]{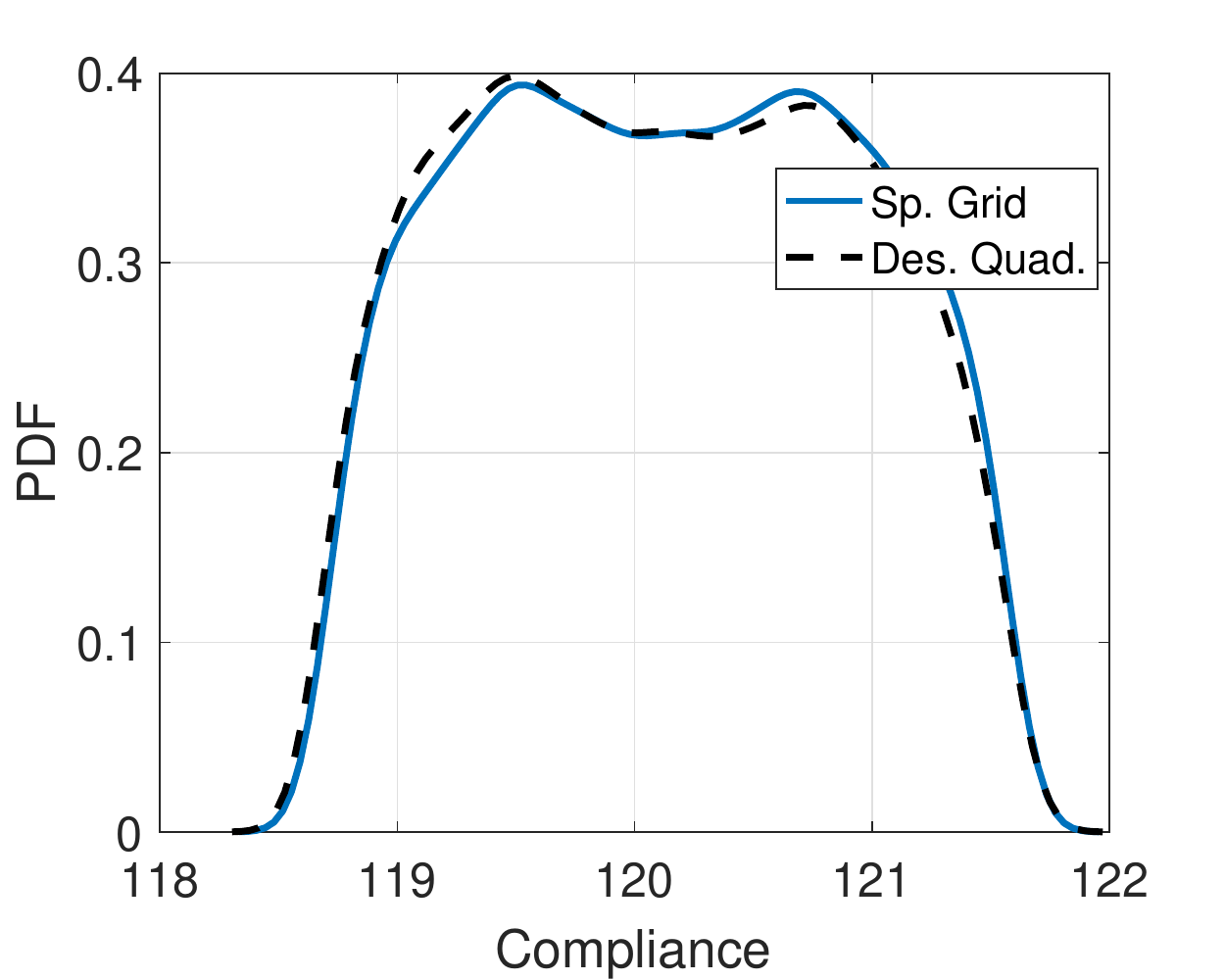}
\end{minipage}
\vspace{-0.3cm}
\caption{\small Robust topology design of the MBB beam (left) Compliance PDF with Sparse Grid and designed quadrature (right) }\label{fig_6_0}
\end{figure}


\begin{table}[!h]
\tiny
\caption{Designed nodes and weights for uniform wight function associated with $d=4,r=6$.}
\centering
\begin{tabular}{c c c c c}
\hline\hline
$x^{(1)}$   & $x^{(2)}$  & $x^{(3)}$  & $x^{(4)}$  &  $w$  \\
\hline
0.257802083101815 &	-0.0703252346579532 &	0.962710865388279	& 0.430231485089995	& 0.0261815176727414\\
-0.816973940130726&	-0.943714906761859	&0.386890523282751	&-0.999999000000035	&0.00580285921526766\\
-0.947032046309044&	0.989213881193871&	0.936667215690650	&-0.993786957614627	&0.00249952956479966\\
-0.410640873236206&	0.954255232732162	&-0.147886263760743	&0.759793319115318	&0.0183482544886740\\
0.231143583536335	&0.304638143131838&	-0.528664154404583	&0.0739711055845102	&0.0410433205157124\\
-0.778491697308234	&0.966653453252213&	-0.383947457004414	&-0.527873684868762	&0.0139368201993277\\
-0.180274497367222&	-0.0792041207370926&	0.828782356307522&	-0.777765421528772&	0.0315300612985717\\
0.540286687837802&	-0.208086450042028	&-0.948367080027728	&-0.501356657306036	&0.0215578327462099\\
-0.683026871793909&	0.647010501071650	&-0.973583715205816	&0.370583426726367	&0.0161108573238146\\
-0.993338263553449&	-0.368646129691905	&-0.737371067829234	&-0.348475820258963	&0.0147573865118333\\
0.407947779669019	&-0.815591287366472	&-0.0796807556349552&	0.378883833053981&	0.0331223595693260\\
0.938583690749431	&-0.673570409626489	&0.750349855298878	&0.515481332659911	&0.0138830007891160\\
0.676361893172788&	-0.0863458892282786	&0.307033851877092	&-0.222719471684161	&0.0666800814123564\\
0.871158904206783&	0.833748485858754	&-0.0988153767941961&	0.190418047822592	&0.0231223581306122\\
-0.637055255430739	&0.561219983742431	&0.650334703413941	&-0.0934633784282683	&0.0464871742390699\\
-0.724596294549971	&-0.469996083934142&	0.597859508917895&	0.576481467431314&	0.0349899149656204\\
-0.0633965531889436&	0.107495869838039&	0.150581945391035&	0.708624486280424&	0.0579979230656018\\
0.447670409694807&	0.676539689861564	&-0.761277790685523	&0.818329189952187	&0.0206811254936501\\
-0.421094936258687&	-0.384610025091731	&-0.645174618242432	&0.192061679511636	&0.0430978935629120\\
0.144084034877935&	0.920034465753694	&0.719442066934779	&0.265836622954943	&0.0226831954639939\\
-0.511575703485635&	-0.940233287027892&	-0.881122093604340	&-0.679836608445536	&0.00914818101224769\\
0.903361192617833&	-0.930618551587704&	-0.556114178416505	&-0.368802948869006	&0.0133824642384102\\
-0.451300661973021&	-0.678204534689092&	0.109872507715082	&-0.268415947407672	&0.0469834603269315\\
0.570179507325998&	0.946048343399397	&-0.888875656537771	&-0.502288128155230	&0.0114411429728491\\
0.947673444496712&	-0.0487093539097375	&-0.745420922954487	&0.483440076541598	&0.0180891170799238\\
-0.151578782518484&	0.536570691044281	&-0.432491091797662	&-0.296255229184200	&0.0429298801103472\\
0.200254358361331&	-0.574573679946025	&-0.370917668694681	&-0.848237515108908	&0.0354514497467619\\
0.703268719144780&	-0.538649441561851	&-0.243402446610143	&0.932031263765523	&0.0178590105949260\\
-0.851069823343954&	-0.999999000000274	&-0.332766663879255	&0.640694186502822	&0.0109542526977389\\
0.778210999922193&	-0.689412923038060&	0.645938727211798	&-0.866284593236453	&0.0141564568596360\\
0.861114502173077&	0.635646385445806	&0.943650909951228	&-0.537424461094664	&0.0122992882684838\\
0.923895085294775&	0.356961947592050	&-0.390843097452625	&-0.911346350395596	&0.0147169035847508\\
0.224936370766468&	0.759504704125777	&0.295151117998272	&-0.810538235949664	&0.0343456718994955\\
0.753566197170202&	0.547890728899287	&0.582098594572620	&0.865739837846072	&0.0201259507621183\\
-0.999996169574511	&0.450595235653401	&-0.0687856135220421&	0.522619830858851&	0.0196754195407364\\
-0.740580150798581	&0.648489590619417	&0.938458459820395	&0.997858479760182	&0.00660961411039667\\
-0.641393244038470	&-0.177020578876251	&-0.678531778799864	&0.993823966856282	&0.0158119117410319\\
0.101688221877749	&-0.925804660338255	&0.769880631707084	&-0.240864991838435	&0.0208587330250274\\
0.183235217131754	&-0.787109851166051	&-0.944143000956805	&0.613293643593588	&0.0158948230283889\\
-0.0665844458312762&	-0.866922211430191&	0.698247471706727&	0.981025573223252&	0.0101514764249415\\
-0.494502217917924&	0.382211470283539	&-0.824748517891121	&-0.935346887279527	&0.0160972432916053\\
-0.973108310610189&	-0.643578784698889	&0.999996903007507	&-0.263921398566453	&0.00715355417527909\\
-0.800132635705771&	0.0390203514319736	&0.0939369967131938	&-0.741463689395994	&0.0313505282787613\\
\hline
\end{tabular}
\label{node_weight_print}
\end{table}

\section{Concluding Remarks}

We present a systematic approach, designed quadrature, for computing multivariate quadrature rules in generic settings. The framework uses penalty methods in constrained optimization to ensure positivity of the weights and feasible locations for the nodes. The Gauss-Newton algorithm is used to perform minimization of the penalty-augmented objective function. $L^2$ regularization is utilized to treat ill-conditioned systems encountered during Newton step updates. On regular domains such as hypercubes, our designed quadrature results in considerably fewer nodes (and guaranteed positive weights) compared to alternative multivariate quadrature rules, such as sparse grids, and hence is promising for computational science and engineering involving expensive simulations.
When applied to a benchmark robust topology optimization problem, designed quadrature reduces requisite cost by nearly half compared with sparse grid rules, and achieves higher accuracy.



\end{document}